\documentclass[11pt,reqno]{article}
\usepackage{amsmath,amsthm,amssymb,cite}
\usepackage{graphicx}
\usepackage{calrsfs}
\usepackage{wrapfig}
\usepackage{tikz}
\usepackage{color}
\usetikzlibrary{calc}
\usepackage{epsfig}
\usetikzlibrary{arrows,chains,matrix,positioning,scopes}
\makeatletter
\tikzset{join/.code=\tikzset{after node path={%
\ifx\tikzchainprevious\pgfutil@empty\else(\tikzchainprevious)%
edge[every join]#1(\tikzchaincurrent)\fi}}}
\makeatother
\tikzset{>=stealth',every on chain/.append style={join},         every join/.style={->}}

\newtheorem{theorem}{Theorem}[section]
\newtheorem{proposition}[theorem]{Proposition}
\newtheorem{lemma}[theorem]{Lemma}
\newtheorem{corollary}[theorem]{Corollary}
\newtheorem{conjecture}[theorem]{Conjecture}
\newtheorem{remark}[theorem]{Remark}
\newtheorem{Definition}[theorem]{Definition}

\title{Forks, Noodles and the Burau Representation\\ for $n=4$}
\author{A. Beridze and P. Traczyk}

\begin{document}

\maketitle

\begin{abstract} The reduced Burau representation is a natural
action of the braid group $B_n$ on the first homology group $H_1({\tilde{D}}_n;\mathbb{Z})$
of a suitable infinite cyclic covering space ${\tilde{D}}_n$ of the $n$--punctured disc $D_n$.  It is known that  the Burau representation is
faithful for $n\le 3$ and that it is not faithful for $n\ge 5$.
We use forks and noodles homological techniques and Bokut--Vesnin generators to analyze the problem for $n=4$.  We present a Conjecture implying faithfulness and a Lemma explaining the implication. We give some arguments suggesting why we expect the Conjecture to be true.  Also, we give some geometrically calculated examples and information about data gathered using a 
C\texttt{++} program.

\end{abstract}

\section{Introduction}

Let us recall the definition of the reduced Burau representation  in terms of the first homology group
$H_1({\tilde{D}}_4;\mathbb{Z})$ of a suitable infinite cyclic covering space ${\tilde{D}}_4$ 
of the $4$-punctured  disc $D_4$.
Let $D_4$ be the unit closed  disc
on the plane with center $(0,0)$ and four punctures at: 
$p_1=\left(-\frac{1}{2},\frac{1}{2}\right)$,
$\ p_2=\left(\frac{1}{2},\frac{1}{2}\right)$,
$p_3=\left(\frac{1}{2},-\frac{1}{2}\right)$,
$\ p_4=\left(-\frac{1}{2},-\frac{1}{2}\right)$
(see Figure~1).

\begin{figure}[ht]
 \centering
\resizebox{150pt}{!}{%
 \begin{tikzpicture}

\filldraw[color=black!60, fill=white, very thick](0,0) circle (2);

\filldraw [gray] (-2,0) circle (1.5pt); 
\filldraw [gray] (-2/3,2/3) circle (1pt);    
\filldraw [gray] (2/3,2/3) circle (1pt);    
\filldraw [gray] (2/3,-2/3) circle (1pt);    
\filldraw [gray] (-2/3,-2/3) circle (1pt);

\tikzstyle{every node}=[font=\large]
\node[right] at (-2, 0) {$p_0$};
\node[above] at (-2/3, 2/3) {$p_1$};
\node[above] at (2/3,2/3) {$p_2$};
\node[above] at (2/3,-2/3) {$p_3$};
\node[above] at (-2/3,-2/3) {$p_4$};

  \end{tikzpicture}
  }
\caption{The $4$--punctured disc with basepoint $p_0$ and puncture points: $p_1$, $p_2$, $p_3$, $p_4$}
\label{figure 1}
\end{figure}
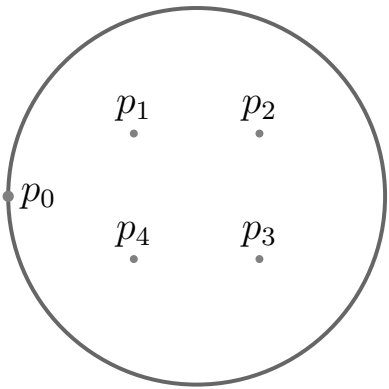

The braid group $B_4$ is the group of all equivalence classes of
orientation preserving homeomorphisms $\varphi :D_4\to D_4$ which
fix the boundary $\partial D_4$ pointwise, where equivalence
relation is isotopy  relative to $\partial D_4$. Let ${\pi }_1\left(D_4\right)$
be the fundamental group of the 4--punctured disc $D_4$ with respect to the basepoint $p_0=(-1,0)$. Consider the map $\varepsilon :{\pi }_1\left(D_4\right) \to \left\langle t\right\rangle $
which sends a loop $\gamma \in {\pi }_1\left(D_4\right)$ to $ t^{\left[\gamma \right]}$, 
where $\left[\gamma \right]$ is the winding number of $\gamma $ around
punctured points $p_1,p_2,p_3,p_4$ (meaning: the sum of the four winding numbers for individual points). 
Let $\pi :{\tilde{D}}_4\to D_4$ be the infinite cyclic 
covering space corresponding to the kernel 
$\ker\left(\varepsilon \right)$ of the map  $\varepsilon :{\pi }_1\left(D_4\right)\to \left\langle t\right\rangle $.
Let ${\tilde{p}}_0$ be any fixed basepoint which is a lift of the basepoint $p_0$. In this case 
$H_1({\tilde{D}}_4;\mathbb{Z})$ is free $\mathbb{Z}\left[t,t^{-1}\right]$--module of rank 3 (see [3]).  Let $\varphi :D_4\to D_4$ be a homeomorphism representing of an element $\sigma \in B_4$. 
It can be lifted to a map $\widetilde{\varphi }:{\tilde{D}}_4\to {\tilde{D}}_4$ which fixes 
the fiber over $p_0.$ Therefore it induces a $\mathbb{Z}\left[t,t^{-1}\right]$--module 
automorphism ${\widetilde{\varphi }}_*:H_1\left({\tilde{D}}_4;\mathbb{Z}\right)\to H_1\left( {\tilde{D}}_4;\mathbb{Z}\right)$. 
Consequently, the reduced Burau representation
$$\rho :B_4\to Aut\left(H_1\left({\tilde{D}}_4;\mathbb{Z}\right)\right)\leqno{(1.1) }$$ 
is given [3] by
$$\rho \left(\sigma \right)={\widetilde{\varphi }}_*,\ \ \forall \sigma \in B_4.\leqno{(1.2) }$$ 

It is known that the Burau representation is
faithful for $n\le 3$ [1], [2] and it is not faithful for $n\ge 5$ [2], [6], [7]. Therefore, the problem is open for $n=4$. 
In this paper, we use the Bokut-Vesnin generators $a,a^{-1},b,b^{-1}$ of a certain free subgroup of $B_4$ (see [4])
and a technique developed in [3], to prove the crucial lemma, which
gives the opportunity to decompose entries $\rho_{11}(a^{n} \sigma )$ and $\rho_{13}(a^{n} \sigma)$ of the Burau matrix $\rho(a^{n} \sigma)$ as a sum of three uniquely determined polynomials and the formula to calculate $\rho_{13}(a^{n+m} \sigma)$ and $\rho (a^{n+m} \sigma)$ polynomials using the given decomposition.   Besides, we formulate Conjecture 4.2, which implies that if a 
non--trivial braid $\sigma \in \ker \rho$ has a certain additional property, then there exists a sufficiently large $l_0$ with respect to the length of $\sigma$ (to be explained in Section~3, Corollary~3.2) and a sufficiently large $m_0$ such that for each $m>m_0$ and $l>l_0$ the difference of lowest degrees of
polynomials $\rho_{13}(a^{n+m} \sigma)$ and $\rho (a^{n+m} \sigma)$ is $-1$. We will present arguments and experimental data showing why we expect the conjecture to be true. Also, we will consider several examples calculated geometrically. We will show that the conjecture implies faithfulness of the Burau representation for $n=4$. 

\section{The Burau representation, Forks and Noodles}

The Burau representation for $n = 4$ was defined by (1.1) and (1.2).
On the other hand $H_1({\tilde{D}}_4;\mathbb{Z})$ is a free $\mathbb{Z}\left[t,t^{-1}\right]$--module of rank 3 and if we take a basis of it, then $Aut\left(H_1\left({\tilde{D}}_4;\mathbb{Z}\right)\right)$ can be identified with $GL\left(3,\ \mathbb{Z}\left[t,t^{-1}\right]\right)$. For this reason we will review the definition of the forks.

\begin{Definition}
A fork is an embedded oriented tree $F$ in the  disc $D$ with four vertices  $p_0,p_i,p_j$ and $z$, 
where $i\neq j, i,j \in\{1,2,3,4\}$ such that (see [3]):

\begin{enumerate}
\item  $F$ meets the puncture points only at $p_i$ and $p_j$;
\item  $F$ meets the boundary $\partial D_4$ only at $p_0$;
\item  All three edges of $F$ have $z$ as a common vertex.
\end{enumerate}
\end{Definition}

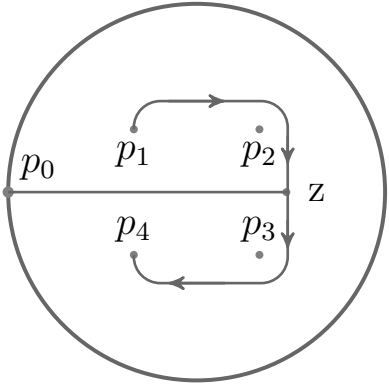
\begin{figure}[ht]
  \centering
\resizebox{150pt}{!}{%
  \begin{tikzpicture}

\filldraw[color=black!60, fill=white, very thick](0,0) circle (2);

\filldraw [gray] (-2,0) circle (1.5pt); 
\filldraw [gray] (-2/3,2/3) circle (1pt);    
\filldraw [gray] (2/3,2/3) circle (1pt);    
\filldraw [gray] (2/3,-2/3) circle (1pt);    
\filldraw [gray] (-2/3,-2/3) circle (1pt);
\filldraw [gray] (20/21,0) circle (1pt);

\tikzstyle{every node}=[font=\large]
\node[above right] at (-2, 0) {$p_0$};
\node[below] at (-2/3, 2/3) {$p_1$};
\node[below] at (2/3,2/3) {$p_2$};
\node[above] at (2/3,-2/3) {$p_3$};
\node[above] at (-2/3,-2/3) {$p_4$};
\node[right] at (21/20,0) {z};
\draw[color=black!60, thick] (-2,0) -- (2/3+0.3,0);
\draw[color=black!60, thick, rounded corners=8pt, -  ] (-2/3,2/3) -- (-2/3,2/3+0.3) - - (2/3+0.3,2/3+0.3) - - (2/3+0.3,-2/3-0.3) - - (-2/3,-2/3-0.3) - - (-2/3,-2/3);
\draw[color=black!60, thick,->] (-0.3,2/3+0.3) - - (0.3,2/3+0.3);
\draw[color=black!60, thick,->] (2/3+0.3,0.6) - - (2/3+0.3,0.3);
\draw[color=black!60, thick,->] (2/3+0.3,-0.3) - - (2/3+0.3,-0.6);
\draw[color=black!60, thick,->] (0.3,-2/3-0.3) - - (-0.3,-2/3-0.3);
  \end{tikzpicture}
  }
  \caption{The line form $p_0$ to $z$ is the handle and the curve from $p_1$ to $p_4$ is the tine $T(F)$ of the fork $F$}
  \label{figure 2}
\end{figure}

The edge of $F$ which contains $p_0$ is called the handle. The union of the other two edges is denoted by $T\left(F\right)$ and it is called tine of  $F$. Orient $T\left(F\right)$ so that the handle of  $F$ lies to the right of  $T\left(F\right)$ (see Figure~2) [3].

For a given fork $F$, let $h:I\to D_4$ be the handle of  $F$, viewed as a path in $D_4$ and take a lift $\tilde{h}:I\to {\tilde{D}}_4$ of $h$ so that $\tilde{h}\left(0\right)={\tilde{p}}_0$. Let $\tilde{T}\left(F\right)$ be the connected component of  ${\pi }^{-1}\left(T\left(F\right)\right)$ which contains the point $\tilde{h}\left(1\right)$. In this case any element of  $H_1\left({\tilde{D}}_4;\mathbb{Z}\right)$ can be viewed as a homology class of $\tilde{T}\left(F\right)$ and it is denoted by $F$ [3].

Standard fork  $F_i,\ \ i=1,2,3\ $ is the fork whose tine edge is the straight arc connecting the i-th and the (i+1)-st punctured points and whose handle has the form as in Figure~3.  It is known that if $F_1,{\ F}_2$ and $F_3$ are the corresponding homology classes, then they form a basis of $H_1\left({\tilde{D}}_4;\mathbb{Z}\right)$ (see [3]).

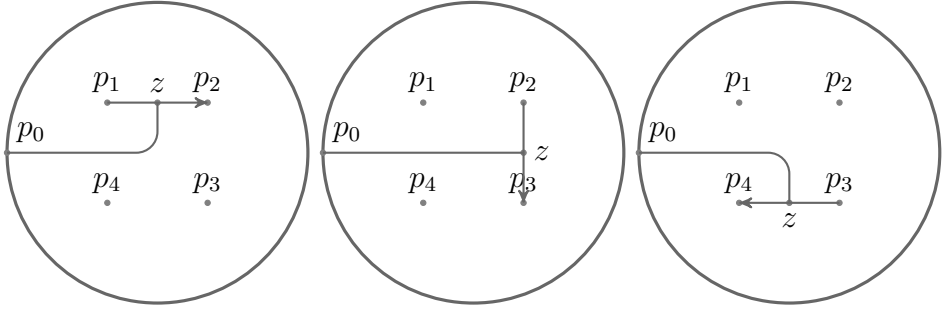
\begin{figure}[ht]
  \centering
  \begin{tikzpicture}

\filldraw[color=black!60, fill=white, very thick](-4.2,0) circle (2);

\filldraw [gray] (-2-4.2,0) circle (1pt); 
\filldraw [gray] (-2/3-4.2,2/3) circle (1pt);    
\filldraw [gray] (2/3-4.2,2/3) circle (1pt);    
\filldraw [gray] (2/3-4.2,-2/3) circle (1pt);    
\filldraw [gray] (-2/3-4.2,-2/3) circle (1pt);
\filldraw [gray] (-4.2,2/3) circle (1pt);

\tikzstyle{every node}=[font=\large]
\node[above right] at (-2-4.2, 0) {$p_0$};
\node[above] at (-2/3-4.2, 2/3) {$p_1$};
\node[above] at (2/3-4.2,2/3) {$p_2$};
\node[above] at (2/3-4.2,-2/3) {$p_3$};
\node[above] at (-2/3-4.2,-2/3) {$p_4$};
\node[above] at (-4.2,2/3) {$z$};

\draw[color=black!60, thick, rounded corners=8pt] (-2-4.2, 0) -- (-4.2,0) -- (-4.2,2/3);
\draw[color=black!60, thick][->] (-2/3-4.2, 2/3) -- (2/3-4.2,2/3);

\filldraw[color=black!60, fill=white, very thick](0,0) circle (2);

\filldraw [gray] (-2,0) circle (1pt); 
\filldraw [gray] (-2/3,2/3) circle (1pt);    
\filldraw [gray] (2/3,2/3) circle (1pt);    
\filldraw [gray] (2/3,-2/3) circle (1pt);    
\filldraw [gray] (-2/3,-2/3) circle (1pt);
\filldraw [gray] (2/3,0) circle (1pt);

\tikzstyle{every node}=[font=\large]
\node[above right] at (-2, 0) {$p_0$};
\node[above] at (-2/3, 2/3) {$p_1$};
\node[above] at (2/3,2/3) {$p_2$};
\node[above] at (2/3,-2/3) {$p_3$};
\node[above] at (-2/3,-2/3) {$p_4$};
\node[right] at (2/3,0) {$z$};

\draw[color=black!60, thick] (-2, 0) -- (2/3,0);
\draw[color=black!60, thick][<-] (2/3, -2/3) -- (2/3,2/3);

\filldraw[color=black!60, fill=white, very thick](4.2,0) circle (2);

\filldraw [gray] (-2+4.2,0) circle (1pt); 
\filldraw [gray] (-2/3+4.2,2/3) circle (1pt);    
\filldraw [gray] (2/3+4.2,2/3) circle (1pt);    
\filldraw [gray] (2/3+4.2,-2/3) circle (1pt);    
\filldraw [gray] (-2/3+4.2,-2/3) circle (1pt);
\filldraw [gray] (4.2,-2/3) circle (1pt);

\tikzstyle{every node}=[font=\large]
\node[above right] at (-2+4.2, 0) {$p_0$};
\node[above] at (-2/3+4.2, 2/3) {$p_1$};
\node[above] at (2/3+4.2,2/3) {$p_2$};
\node[above] at (2/3+4.2,-2/3) {$p_3$};
\node[above] at (-2/3+4.2,-2/3) {$p_4$};
\node[below] at (4.2,-2/3) {$z$};

\draw[color=black!60, thick, rounded corners=8pt] (-2+4.2, 0) -- (4.2,0)-- (4.2,-2/3);

\draw[color=black!60, thick] [<-] (-2/3+4.2, -2/3) -- (2/3+4.2,-2/3);

  \end{tikzpicture}
  \caption{Standard forks:  $F_1,~~F_2,~~F_3$}
  \label{figure 3}
\end{figure}

\noindent
Using the basis derived from $F_1,F_2,F_3$, any automorphism ${\widetilde{\varphi }}_*:H_1\left({\tilde{D}}_4;\mathbb{Z}\right)\to H_1\left({\tilde{D}}_4;\mathbb{Z}\right)$ can be viewed as a $3\times 3$ matrix with elements in the free $\mathbb{Z}\left[t,t^{-1}\right]$--module [3]. If $\varphi :D_4\to D_4$ is representing an element $\sigma \in B_4$,
then we need to write the matrix $\rho \left(\sigma \right)={\widetilde{\varphi }}_*$ in terms of homology (algebraic) intersection pairing
\[\left\langle -,-\right\rangle :H_1\left({\tilde{D}}_4;\mathbb{Z}\right)\times H_1\left({\tilde{D}}_4,\partial {\tilde{D}}_4;\mathbb{Z}\right)\to \mathbb{Z}\left[t,t^{-1}\right].\] 
For this aim we need to define the noodles which represent relative homology classes in $H_1\left({\tilde{D}}_4,\partial {\tilde{D}}_4;\mathbb{Z}\right)$. 
\begin{Definition}
A noodle is an embedded oriented arc in $D_4$, which begins at the base point $p_0$ and ends at some point of the boundary $\partial D_4$ [3]. 
\end{Definition}
For each $a\in H_1\left({\tilde{D}}_4;\mathbb{Z}\right)$ and $b\in H_1\left({\tilde{D}}_4,\partial {\tilde{D}}_4;\mathbb{Z} \right)$ we should take the corresponding fork $F$ and noodle $N$ and define the polynomial $\left\langle F,N\right\rangle \in \mathbb{Z}\left[t,t^{-1}\right]$. It does not depend on the choice of representatives of homology classes and so  
\[\left\langle -,-\right\rangle :H_1\left({\tilde{D}}_4;\mathbb{Z} \right)\times H_1\left({\tilde{D}}_4,\partial {\tilde{D}}_4;\mathbb{Z} \right)\to  \mathbb{Z} \left[t,t^{-1}\right]\] 
is well-defined [3]. The map defined by the above formula is called  the noodle--fork paring. Note that geometrically it can be computed in the following way: Let $F$ be a fork and $N$ be a noodle, such that $T\left(F\right)$ intersects $N$ transversely. Let $z_1,z_2,\ \dots ,z_n$ be the intersection points. For each point $z_i$ let ${\varepsilon }_i$ be the sign of the intersection between $T\left(F\right)$ and $N$ at $z_i$ (the intersection is positive if going from tine to noodle according to the chosen directions means turning left)  and $e_i=\left[{\gamma }_i\right]$ be the winding number of the loop ${\gamma }_i$  around the puncture points $p_1,p_2,p_3,p_4$, where  ${\gamma }_i$ is the composition of three paths $h$, $t_i$ and $n_i$:

\begin{enumerate}
\item $h$ is a path from $p_0$ to $z$ along the handle of  $F$ (see Figure 4a);

\item $t_i$ is a path from $z$ to $z_i$ along the tine  $T(F)$ (see Figure 4b);

\item $n_i$ is a path from $z_i$ to  $p_0$ along the noodle  $N$ (see Figure 4c).
\end{enumerate}

\begin{figure}[ht]
  \centering
  \begin{tikzpicture}

\filldraw[color=black!60, fill=white, very thick](-4.2,0) circle (2);

\draw[color=red!60, thick, rounded corners=10pt, -> ] (-2-4.2,0) -- (0-4.2, 0) - - (0-4.2, 2);

\draw[color=blue!100, thick, rounded corners=5pt, - ]  (-2-4.2,0) - - (-2-4.2,-0.1) -- (2/3+0.1-4.2, -0.1) - - (2/3+0.1-4.2, 0.2) -- (2/3+0.1-4.2, 2/3+0.2) - - (2/3-0.2-4.2, 2/3+0.2) - - (2/3-0.2-4.2, 2/3);
\draw[color=black!60, thick,  -> ] (2/3-4.2,2/3) -- (-2/3-4.2,2/3);

\filldraw [gray] (-2-4.2,0) circle (1pt); 
\filldraw [gray] (-2/3-4.2,2/3) circle (1pt);    
\filldraw [gray] (2/3-4.2,2/3) circle (1pt);    
\filldraw [gray] (2/3-4.2,-2/3) circle (1pt);    
\filldraw [gray] (-2/3-4.2,-2/3) circle (1pt);
\filldraw [gray] (2/3-0.2-4.2, 2/3) circle (2pt); 
\filldraw [gray] (0-4.2, 2/3) circle (2pt);

\tikzstyle{every node}=[font=\large]
\node[right] at (-2-4.2, 0) {$p_0$};
\node[above] at (-2/3-4.2, 2/3) {$p_1$};
\node[above] at (2/3-4.2,2/3) {$p_2$};
\node[above] at (2/3-4.2,-2/3) {$p_3$};
\node[above] at (-2/3-4.2,-2/3) {$p_4$};
\node[below] at (2/3-0.2-4.2, 2/3) {$z$}; 
\node[below] at (0-4.2, 2/3) {$z_1$};
\node[above] at (0.3-4.2, -0.2) {$h$};
\node[above] at (-4.2, -2.5) {$a$};

\filldraw[color=black!60, fill=white, very thick](0,0) circle (2);

\draw[color=red!60, thick, rounded corners=10pt, -> ] (-2,0) -- (0, 0) - - (0, 2);

\draw[color=black!60, thick, rounded corners=5pt, - ]  (-2,0) - - (-2,-0.1) -- (2/3+0.1, -0.1) - - (2/3+0.1, 0.2) -- (2/3+0.1, 2/3+0.2) - - (2/3-0.2, 2/3+0.2) - - (2/3-0.2, 2/3);
\draw[color=black!60, thick,  - ] (2/3,2/3) -- (2/3-0.2, 2/3);
\draw[color=blue!100, thick,  -  ] (2/3-0.2, 2/3) -- (0, 2/3);
\draw[color=black!60, thick,  -> ] (0, 2/3) -- (-2/3,2/3);

\filldraw [gray] (-2,0) circle (1pt); 
\filldraw [gray] (-2/3,2/3) circle (1pt);    
\filldraw [gray] (2/3,2/3) circle (1pt);    
\filldraw [gray] (2/3,-2/3) circle (1pt);    
\filldraw [gray] (-2/3,-2/3) circle (1pt);
\filldraw [gray] (2/3-0.2, 2/3) circle (2pt); 
\filldraw [gray] (0, 2/3) circle (2pt);

\tikzstyle{every node}=[font=\large]
\node[right] at (-2, 0) {$p_0$};
\node[above] at (-2/3, 2/3) {$p_1$};
\node[above] at (2/3,2/3) {$p_2$};
\node[above] at (2/3,-2/3) {$p_3$};
\node[above] at (-2/3,-2/3) {$p_4$};
\node[below] at (2/3-0.2, 2/3) {$z$}; 
\node[below] at (0, 2/3) {$z_1$};
\node[above] at (0.3, 2/3-0.1) {$t_1$};
\node[above] at (0, -2.5) {$b$};

\filldraw[color=black!60, fill=white, very thick](4.2,0) circle (2);

\draw[color=blue!60, thick, rounded corners=10pt, - ] (-2+4.2,0) -- (0+4.2, 0) - - (0+4.2, 2/3);
\draw[color=red!60, thick, rounded corners=10pt, -> ] (0+4.2, 2/3) - - (0+4.2, 2);

\draw[color=black!60, thick, rounded corners=5pt, - ]  (-2+4.2,0) - - (-2+4.2,-0.1) -- (2/3+0.1+4.2, -0.1) - - (2/3+0.1+4.2, 0.2) -- (2/3+0.1+4.2, 2/3+0.2) - - (2/3-0.2+4.2, 2/3+0.2) - - (2/3-0.2+4.2, 2/3);
\draw[color=black!60, thick,  -> ] (2/3+4.2,2/3) -- (-2/3+4.2,2/3);

\filldraw [gray] (-2+4.2,0) circle (1pt); 
\filldraw [gray] (-2/3+4.2,2/3) circle (1pt);    
\filldraw [gray] (2/3+4.2,2/3) circle (1pt);    
\filldraw [gray] (2/3+4.2,-2/3) circle (1pt);    
\filldraw [gray] (-2/3+4.2,-2/3) circle (1pt);
\filldraw [gray] (2/3-0.2+4.2, 2/3) circle (2pt); 
\filldraw [gray] (0+4.2, 2/3) circle (2pt);

\tikzstyle{every node}=[font=\large]
\node[right] at (-2+4.2, 0) {$p_0$};
\node[above] at (-2/3+4.2, 2/3) {$p_1$};
\node[above] at (2/3+4.2,2/3) {$p_2$};
\node[above] at (2/3+4.2,-2/3) {$p_3$};
\node[above] at (-2/3+4.2,-2/3) {$p_4$};
\node[below] at (2/3-0.2+4.2, 2/3) {$z$}; 
\node[below] at (0+4.2, 2/3) {$z_1$};
\node[above] at (-0.5+4.2, -0.1) {$n_1$};
\node[above] at (4.2, -2.5) {$c$};

  \end{tikzpicture}
  \caption{h -- a path from $p_0$ to $z$; $t_1$ -- a path from $z$ to $z_1$; $n_1$ -- a path from $z_1$ to $p_0$}
  \label{figure 4}
\end{figure}
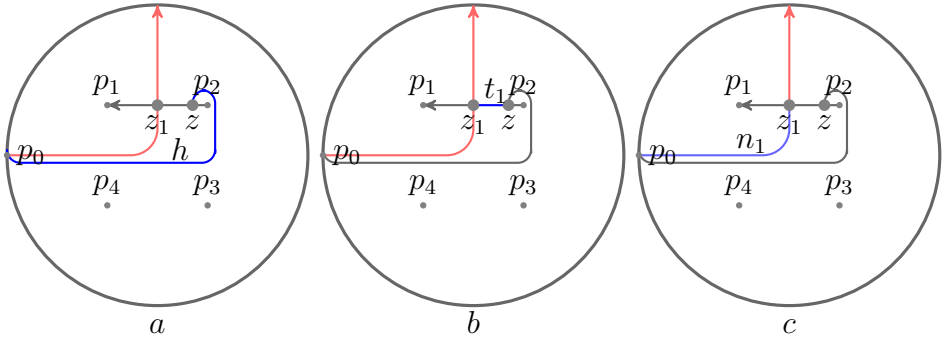

\noindent In such case the noodle-fork pairing of $F$ and $N$ is given by (see [3]):
$$\left\langle F,N\right\rangle =\sum_{1\le i\le n}{{\varepsilon }_it^{e_i}}\in 
\mathbb{Z} \left[t,t^{-1}\right]. \leqno{(2.1) }$$ 

Let $N_1$, $N_2$ and $N_3$ be the noodles given in Figure~5. These are called standard noodles. For each braid $\sigma \in B_4$, the corresponding Burau matrix $\rho \left(\sigma \right)$ can be computed using noodle-fork pairing of standard noodles and standard forks. In particular the following is true.

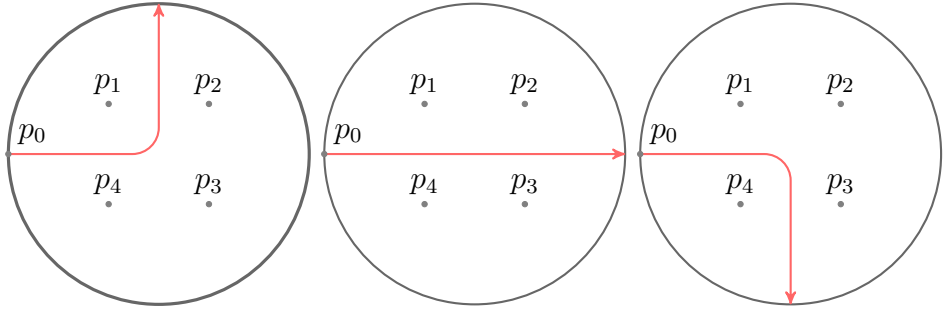
\begin{figure}[ht]
  \centering
  \begin{tikzpicture}

\filldraw[color=black!60, fill=white, very thick](-4.2,0) circle (2);

\draw[color=red!60, thick, rounded corners=10pt, -> ] (-2-4.2,0) -- (-4.2, 0) - - (-4.2, 2);

\filldraw [gray] (-2-4.2,0) circle (1pt); 
\filldraw [gray] (-2/3-4.2,2/3) circle (1pt);    
\filldraw [gray] (2/3-4.2,2/3) circle (1pt);    
\filldraw [gray] (2/3-4.2,-2/3) circle (1pt);    
\filldraw [gray] (-2/3-4.2,-2/3) circle (1pt);

\tikzstyle{every node}=[font=\large]
\node[above right] at (-2-4.2, 0) {$p_0$};
\node[above] at (-2/3-4.2, 2/3) {$p_1$};
\node[above] at (2/3-4.2,2/3) {$p_2$};
\node[above] at (2/3-4.2,-2/3) {$p_3$};
\node[above] at (-2/3-4.2,-2/3) {$p_4$};

\filldraw[color=black!60, fill=white, thick](0,0) circle (2);

\draw[color=red!60, thick, ->] (-2,0) - - (2,0);

\filldraw [gray] (-2,0) circle (1pt); 
\filldraw [gray] (-2/3,2/3) circle (1pt);    
\filldraw [gray] (2/3,2/3) circle (1pt);    
\filldraw [gray] (2/3,-2/3) circle (1pt);    
\filldraw [gray] (-2/3,-2/3) circle (1pt);

\tikzstyle{every node}=[font=\large]
\node[above right] at (-2, 0) {$p_0$};
\node[above] at (-2/3, 2/3) {$p_1$};
\node[above] at (2/3,2/3) {$p_2$};
\node[above] at (2/3,-2/3) {$p_3$};
\node[above] at (-2/3,-2/3) {$p_4$};

\filldraw[color=black!60, fill=white, thick](4.2,0) circle (2);

\draw[color=red!60, thick, rounded corners=10pt, - > ] (-2+4.2,0) - - (4.2, 0) - - (4.2, -2);

\filldraw [gray] (-2+4.2,0) circle (1pt); 
\filldraw [gray] (-2/3+4.2,2/3) circle (1pt);    
\filldraw [gray] (2/3+4.2,2/3) circle (1pt);    
\filldraw [gray] (2/3+4.2,-2/3) circle (1pt);    
\filldraw [gray] (-2/3+4.2,-2/3) circle (1pt);

\tikzstyle{every node}=[font=\large]
\node[above right] at (-2+4.2, 0) {$p_0$};
\node[above] at (-2/3+4.2, 2/3) {$p_1$};
\node[above] at (2/3+4.2,2/3) {$p_2$};
\node[above] at (2/3+4.2,-2/3) {$p_3$};
\node[above] at (-2/3+4.2,-2/3) {$p_4$};

  \end{tikzpicture}
  \caption{Standard noodles: $N_1,~~N_2,~~N_3$}
  \label{figure 5}
\end{figure}

\begin{lemma}(see [5]). Let $\sigma \in B_n$. Then for $1\le i,j\le n-1$, the entry
 ${\rho }_{ij}\left(\sigma \right)$ of its Burau matrix $\rho \left(\sigma \right)$ is 
given by 
\[{\rho }_{ij}\left(\sigma \right)=\left\langle F_i\sigma ,N_j\right\rangle .\] 
\end{lemma}

Note that under the convention adopted here we have
\[\rho \left({\sigma }_1\right)=\left( \begin{array}{ccc}
-t^{-1} & 0 & 0 \\ 
t^{-1} & 1 & 0 \\ 
0 & 0 & 1 \end{array}
\right),     \rho \left({\sigma }_2\right)=\left( \begin{array}{ccc}
1 & 1 & 0 \\ 
0 & -t^{-1} & 0 \\ 
0 & t^{-1} & 1 \end{array}
\right)\]\\
and \[ \rho \left({\sigma }_3\right)=\left( \begin{array}{ccc}
1 & 0 & 0 \\ 
0 & 1 & 1 \\ 
0 & 0 & -t^{-1} \end{array}
\right).\] 
\\\

For example, to calculate $\rho_{1,1}(\sigma_1)$ entry of the matrix $\rho(\sigma_1)$ see the corresponding Figure 6. Note that intersection of the tine $T\left( F_1 \sigma_1 \right)$ of the fork $ F_1 \sigma_1 $ and the noodle $N_1$ 
 at point $z_1$ is negative which means that ${\varepsilon }_i = -1$  (see Figure 6a). On the other hand the winding number $e_1$ of the loop ${\gamma }_1$  (see Figure 6b) around puncture points equals $-1$
because the considered loop misses $p_1, p_3$ and $p_4$ and it goes around $p_2$ once in anti--clockwise direction. Therefore
$$\rho_{11}(\sigma_1)=\left\langle F_1\sigma_1 ,N_1\right\rangle=-t^{-1}$$

 \ \
 
 \ \
 
 \ \
 
 \ \

\begin{figure}[ht]
  \centering
\resizebox{330pt}{!}{%
  \begin{tikzpicture}

\filldraw[color=black!60, fill=white, very thick](-2.5,0) circle (2);

\draw[color=red!60, thick, rounded corners=10pt, -> ] (-2-2.5,0) -- (-2.5, 0) - - (-2.5, 2);

\draw[color=black!60, thick, rounded corners=5pt, - ]  (-2-2.5,0) - - (-2-2.5,-0.1) -- (2/3+0.1-2.5, -0.1) - - (2/3+0.1-2.5, 0.2) -- (2/3+0.1-2.5, 2/3+0.2) - - (2/3-0.2-2.5, 2/3+0.2) - - (2/3-0.2-2.5, 2/3);
\draw[color=black!60, thick,  -> ] (2/3-2.5,2/3) -- (-2/3-2.5,2/3);

\filldraw [gray] (-2-2.5,0) circle (1pt); 
\filldraw [gray] (-2/3-2.5,2/3) circle (1pt);    
\filldraw [gray] (2/3-2.5,2/3) circle (1pt);    
\filldraw [gray] (2/3-2.5,-2/3) circle (1pt);    
\filldraw [gray] (-2/3-2.5,-2/3) circle (1pt);
\filldraw [gray] (2/3-0.2-2.5, 2/3) circle (2pt); 
\filldraw [gray] (0-2.5, 2/3) circle (2pt);

\tikzstyle{every node}=[font=\large]
\node[right] at (-2-2.5, 0) {$p_0$};
\node[above] at (-2/3-2.5, 2/3) {$p_1$};
\node[above] at (2/3-2.5,2/3) {$p_2$};
\node[above] at (2/3-2.5,-2/3) {$p_3$};
\node[above] at (-2/3-2.5,-2/3) {$p_4$};
\node[below] at (2/3-0.2-2.5, 2/3) {$z$}; 
\node[below] at (0-2.5, 2/3) {$z_1$};
\node[below] at (0-2.5, -2.5) {$a: e_1=-1$};

\filldraw[color=black!60, fill=white, very thick](2.5,0) circle (2);

\draw[color=red!60, thick, rounded corners=10pt, - ] (0+2.5, 2/3) -- (0+2.5, 0) - - (-2+2.5,0);
\draw[color=red!60, thick, rounded corners=10pt, -> ] (-2/3+2.5+0.3, 0) -- (-2/3+2.5-0.3, 0);

\draw[color=black!60, thick, rounded corners=5pt, -> ]  (-2/3+2.5-0.3,-0.1) -- (-2/3+2.5+0.3, -0.1);

\draw[color=black!60, thick, rounded corners=5pt, - ]  (-2+2.5,0) - - (-2+2.5,-0.1) -- (2/3+0.1+2.5, -0.1) - - (2/3+0.1+2.5, 0.2) -- (2/3+0.1+2.5, 2/3+0.2) - - (2/3-0.2+2.5, 2/3+0.2) - - (2/3-0.2+2.5, 2/3);

\draw[color=black!60, thick,  -> ] (2/3-0.2+2.5,2/3) -- (0+2.5,2/3);
\draw[color=black!60, thick,  -> ] (2/3-0.2+2.5,2/3) -- (0+2.5,2/3);

\filldraw [gray] (-2+2.5,0) circle (1pt); 
\filldraw [gray] (-2/3+2.5,2/3) circle (1pt);    
\filldraw [gray] (2/3+2.5,2/3) circle (1pt);    
\filldraw [gray] (2/3+2.5,-2/3) circle (1pt);    
\filldraw [gray] (-2/3+2.5,-2/3) circle (1pt);
\filldraw [gray] (2/3-0.2+2.5, 2/3) circle (2pt); 
\filldraw [gray] (0+2.5, 2/3) circle (2pt);

\tikzstyle{every node}=[font=\large]
\node[right] at (-2+2.5, 0) {$p_0$};
\node[above] at (-2/3+2.5, 2/3) {$p_1$};
\node[above] at (2/3+2.5,2/3) {$p_2$};
\node[above] at (2/3+2.5,-2/3) {$p_3$};
\node[above] at (-2/3+2.5,-2/3) {$p_4$};
\node[below] at (2/3-0.2+2.5, 2/3) {$z$}; 
\node[below] at (0+2.5, 2/3) {$z_1$};
\node[above] at (0.3+2.5, -0.2) {$h$};
\node[above] at (0.3+2.5, 2/3-0.1) {$t_1$};
\node[above] at (-0.5+2.5, -0.1) {$n_1$};

\node[below] at (0+2.5, -2.5) {$b: \gamma_1 = n_1*t_1*h;  \varepsilon_1 = -1$};

  \end{tikzpicture}
  }
  \caption{$\left\langle F_1\sigma_1 ,N_1\right\rangle=-t^{-1}$ is monomial, because the intersection of $ F_1 \sigma_1 $ and the  noodle $N_1$ is just one point $z_1$ }
  \label{figure 6}
\end{figure}

\section{The Bokut-Vesnin generators and kernel elements of the Burau representation}

The braid groups $B_4$ and $B_3$ are defined by the following standard presentations [1]: 
\[B_4=\left\langle {\sigma }_1,\ {\sigma }_2,{\sigma }_3\left|{\sigma }_1{\sigma }_2{\sigma }_1={\sigma }_2{\sigma }_1{\sigma }_2,\ \ {\ \sigma }_3{\sigma }_2{\sigma }_3={\sigma }_2{\sigma }_3{\sigma }_2,\ \ \ {\sigma }_3{\sigma }_1={\sigma }_1{\sigma }_3\right.\right\rangle ,\] 
\[B_3=\left\langle {\sigma }_1,\ {\sigma }_2\left|{\sigma }_1{\sigma }_2{\sigma }_1={\sigma }_2{\sigma }_1{\sigma }_2\right.\right\rangle .\] 
Let $\varphi :B_4\to B_3$ be the homomorphism defined by
\[\varphi \left({\sigma }_1\right)={\sigma }_1,\ \ \varphi \left({\sigma }_2\right)={\sigma }_2,\ \ \varphi \left({\sigma }_3\right)={\sigma }_1.\] 
The kernel of  $\varphi $ is known to be a free group $F\left( a,b\right)$ of two generators [4];
\[a={\sigma }_1{\sigma }_2{\sigma }^{-1}_1{\sigma }_3{\sigma }_2^{-1}{\sigma }_1^{-1},\ \ b={\sigma }_3{\sigma }^{-1}_1.\] 

This was proved by L. Bokut and A. Vesnin [4]. We will refer to $a$ and $b$ as the Bokut--Vesnin generators.  
The generators $a$ and $b$ are in fact much more similar than they look at the first glance. This becomes obvious when we interpret $B_4$ as the mapping class group of the $4$--punctured disc. In this well--known approach a braid is an isotopy class of homeomorphisms of the punctured disc fixing the boundary. Figure~7 shows $a$ and $b$ as homeomorphisms of the  punctured disc. The punctures are arranged to make the similarity more visible. Another advantage of this approach is that 
it gives natural interpretation to various actions of $B_4$ to be considered later in this paper.

\begin{figure}[ht]
  \centering
  \resizebox{330pt}{!}{%
  \begin{tikzpicture}

\filldraw[color=black!60, fill=white, very thick](-2.5,0) circle (2);

\draw[color=black!60, fill=white, thick,->] (-2/3+0.2-2.5, 2/3) arc (15:-15:2.5);
\draw[color=black!60, fill=white, thick,->] (2/3+0.2-2.5, -2/3) arc (-15:15:2.5);

\draw[color=black!60, fill=white, thick,->] (-2/3-0.2-2.5, -2/3) arc (195:165:2.5);
\draw[color=black!60, fill=white, thick,->] (2/3-0.2-2.5, 2/3) arc (165:195:2.5);

\filldraw [gray] (-2-2.5,0) circle (1pt); 
\filldraw [gray] (-2/3-2.5,2/3) circle (1pt);    
\filldraw [gray] (2/3-2.5,2/3) circle (1pt);    
\filldraw [gray] (2/3-2.5,-2/3) circle (1pt);    
\filldraw [gray] (-2/3-2.5,-2/3) circle (1pt);

\tikzstyle{every node}=[font=\large]
\node[right] at (-2-2.5, 0) {$p_0$};
\node[above] at (-2/3-2.5, 2/3) {$p_1$};
\node[above] at (2/3-2.5,2/3) {$p_2$};
\node[above] at (2/3-2.5,-2/3) {$p_3$};
\node[above] at (-2/3-2.5,-2/3) {$p_4$};

\filldraw[color=black!60, fill=white, very thick](2.5,0) circle (2);

\draw[color=black!60, fill=white, thick,->] (-2/3+2.5, 2/3+0.2) arc (105:75:2.5);
\draw[color=black!60, fill=white, thick,->] (-2/3+2.5, -2/3-0.2) arc (-105:-75:2.5);

\draw[color=black!60, fill=white, thick,->] (2/3+2.5, 2/3-0.2) arc (-75:-105:2.5);
\draw[color=black!60, fill=white, thick,->] (2/3+2.5, -2/3+0.2) arc (75:105:2.5);

\filldraw [gray] (-2+2.5,0) circle (1pt); 
\filldraw [gray] (-2/3+2.5,2/3) circle (1pt);    
\filldraw [gray] (2/3+2.5,2/3) circle (1pt);    
\filldraw [gray] (2/3+2.5,-2/3) circle (1pt);    
\filldraw [gray] (-2/3+2.5,-2/3) circle (1pt);

\tikzstyle{every node}=[font=\large]
\node[right] at (-2+2.5, 0) {$p_0$};
\node[right] at (-2/3+2.5, 2/3) {$p_1$};
\node[right] at (2/3+2.5,2/3) {$p_2$};
\node[right] at (2/3+2.5,-2/3) {$p_3$};
\node[right] at (-2/3+2.5,-2/3) {$p_4$};

  \end{tikzpicture}
  }
  \caption{}
  \label{figure 7}
\end{figure}

The following Proposition is crucial to our considerations.

\begin{proposition} $\ker \rho_4 \subset \ker \varphi$
\end{proposition}

\begin{proof}
Let us make a slight detour into the realm of the Temperley-Lieb algebras $TL_3$ and $TL_4$. The Temperley-Lieb algebra $TL_n$ is defined as an algebra over $\mathbb{Z}[t,t^{-1}]$. It has $n-1$ generators $\lbrace U^i_i \rbrace _{i=1}^{n-1}$, and the following relations:

(TL1) $U^i_i  U^i_i=(-t^{-2}-t^2)U_i^i,$

(TL2) $U^i_i  U^j_j  U^i_i=U_i^i,$ for $|i-j|=1,$

(TL3) $U^i_i  U^j_j=U^j_j  U^i_i,$ for $|i-j|>1. $

Let us consider the homomorphism $\psi :TL_4 \to TL_3 $ defined by
\[U_1^1 \to U_1^1, \\\  U_2^2 \to U_2^2, \\\ U_3^3 \to U_1^1.\] 

Also, we need to use the Jones' representation $\theta :B_n \to TL_n$ defined by sending $\sigma_i$ to $A+A^{-1}U^i_i$. It is known (see [3], Proposition~1.5) that for $n=3,4$ we have $\ker \theta_n = \ker \rho_n$.
Moreover, the following diagram is obviously commutative:

\begin{center}
\begin{tikzpicture}
\node (A) {$B_4$};
\node (B) [node distance=3cm, right of=A] {$TL_4$};
\node (C) [node distance=2cm, below of=A] {$B_3$};
\node (D) [node distance=3cm, right of=C] {$TL_3$};
\draw[->] (A) to node [above]{$ \theta_4 $} (B);
\draw[->] (A) to node [left]{$ \varphi $}(C);
\draw[->] (C) to node [below]{$ \theta_3$} (D);
\draw[->] (B) to node [right]{$ \psi $} (D);
\end{tikzpicture}
\end{center}

On the other hand the representation $\theta_3$ is faithful and therefore $\ker \rho_4  = \ker \theta_4 \subset \ker \varphi$.
\end{proof}

\begin{corollary} All kernel elements of the Burau representation may be written as words in the Bokut--Vesnin generators $a$, $b$, $a^{-1}$, $b^{-1}$. Moreover, all possible $nontrivial$ elements in the kernel may be written as reduced words of positive  length.
\end{corollary}
\noindent We will use this fact in the next section.

We present for future use the images of $a$, $b$, $a^{-1}$ and  $b^{-1}$ under the Burau representation:

\[\rho \left(a\right)=\left( \begin{array}{ccc}
-t^{-1}+1 & -t^{-1}+t & -t^{-1} \\ 
0 & -t & 0 \\ 
-1 & 0 & 0 \end{array}
\right),   \] 

\[\rho \left(b\right)=\left( \begin{array}{ccc}
-t & 0 & 0 \\ 
1 & 1 & 1 \\ 
0 & 0 & -t^{-1} \end{array}
\right).\] 

\[\rho \left(a^{-1}\right)=\left( \begin{array}{ccc}
0 & 0 & -1 \\ 
0 & -t^{-1} & 0 \\ 
-t & t^{-1}-t & 1-t \end{array}
\right),   \] 

\[\rho \left(b^{-1}\right)=\left( \begin{array}{ccc}
-t^{-1} & 0 & 0 \\ 
t^{-1} & 1 & t \\ 
0 & 0 & -t \end{array}
\right).\]

\section{Faithfulness Problem of the Burau representation}

Let us outline the strategy for analyzing $\ker \rho_4$ in general terms. Consider a braid $\sigma$ that is
a candidate for a non--trivial kernel element of  the Burau representation. Of course we can exclude from our considerations all those non--trivial braids for which we 
{\sl know} for whatever reason that they do not belong to the kernel. Also, 
we can adjust the remaining candidates in some ways --- like replacing $\sigma$ with 
a suitably chosen conjugate of $\sigma$. For such a suitably chosen braid $\sigma$ we need to give some argument which shows that $\rho_{11}(\sigma)$ and $\rho_{31}(\sigma)$ should be non--zero and that $\deg_{min}(\rho_{11}) - \deg_{min}(\rho_{31}) = -1$,
 where  $\deg_{min}$ denotes the exponent of the lowest degree term in the considered Laurent polynomial.

\bigskip
To simplify notation we will denote by
$S_i(t^{\pm 1})$ the $i^{th}$ partial sum of the geometric series with initial term $1$ and quotient 
$-t$ or $-t^{-1}$ (e.g. $S_2(t^{-1}) = 1 - t^{-1} + t^{-2}$).

\begin{lemma} For each braid $\sigma \in B_4$ there exists $n\in \mathbb{N}$, such that\\
\noindent (1) the ${\rho }_{11}(a^{n}\sigma )$ and ${\rho }_{31}(a^{n}\sigma )$  
entries of the Burau matrix $\rho \left(a^n\sigma \right)$ can be decomposed as a sum of three uniquely
determined polynomials
\[{\rho }_{11}\left(a^n\sigma \right)=P\left(t,t^{-1}\right)\left(1-t^{-1}\right)+Q\left(t,t^{-1}\right)+R\left(t,t^{-1}\right)(1-t),\] 
\[{\ \rho }_{31}\left(a^n\sigma \right)=-P\left(t,t^{-1}\right)-Q\left(t,t^{-1}\right)-R\left(t,t^{-1}\right).\] 
\\
such that\\
\noindent (2) for each $m\in \mathbb{N}$ we have
\[{\rho }_{11}\left(a^{m+n}\sigma \right)=P\left(t,t^{-1}\right)
\left(S_{m+1}(t^{-1})\right)+Q \left( t,t^{-1} \right)+ R \left( t,t^{-1} \right) \left( S_{m+1} (t) \right),\] 
\[{\rho }_{31}\left(a^{m+n}\sigma \right)=-P \left(t,t^{-1}\right) \left(S_m(t^{-1})\right)- 
Q\left(t,t^{-1}\right)-R\left(t,t^{-1}\right)\left(S_m(t)\right).\] 
\\
\noindent (3) Moreover, if $\sigma$ is a pure braid, then the polynomial $P$ is non--zero.
\end{lemma}

\begin{proof} First of all let us observe that uniqueness of $P,Q$ and $R$ follows 
from properties (1) and (2) and general algebra. This means that we only need to prove existence
and property (3). While it is possible to give specific algebraic formulas for $P,Q$ and $R$ 
we prefer to prove existence using forks and noodles. We will always assume that the fork/noodle 
configuration considered is irreducible.\\
Let $\sigma \in B_4$ be any braid.
By Lemma~2.3 ${\rho }_{11}\left(a^n\sigma \right)=\left\langle F_1a^n\sigma ,N_1\right\rangle $ 
and ${\rho }_{31}\left(a^n\sigma \right)=\left\langle F_3a^n\sigma ,N_1\right\rangle $. On the other 
hand $\left\langle -,-\right\rangle $
is a bilinear form, so 
$\left\langle F_1a^n\sigma ,N_1\right\rangle =\left\langle F_1a^n,N_1{\sigma }^{-1}\right\rangle $ 
and $\left\langle F_3a^n\sigma ,N_1\right\rangle =\left\langle F_3a^n,N_1{\sigma }^{-1}\right\rangle $.
It follows that
 \[{\rho }_{11}\left(a^n\sigma \right)=\left\langle F_1a^n,N_1{\sigma }^{-1}\right\rangle ,\] 
\[{\rho }_{31}\left(a^n\sigma \right)=\left\langle F_3a^n,N_1{\sigma }^{-1}\right\rangle .\]
 

Let us consider $N_1{\sigma }^{-1}$, the image of the standard noodle $N_1$ under the action of $\sigma^{-1}$.
$N_1{\sigma }^{-1}$ is a path in $D_4$ that begins at the base 
point $p_0$ and ends at the point $\left(0,1\right)\in \partial D_4$. By the definition of 
the standard noodle $N_1$ it is clear that $N_1{\sigma }^{-1}$ divides  $D_4$ into two 
components, such that there is one puncture point in one component and three puncture points in the other.
Let us assume that the single point is $p_1$. For example see Figure~8.
\begin{figure}[ht]
  \centering
 \resizebox{150pt}{!}{%
  \begin{tikzpicture}

\filldraw[color=black!60, fill=white, very thick](0,0) circle (2);

\draw[color=red!60, thick,rounded corners=3pt](-2, 0) - - (-2/3-0.3,0) - - (-2/3-0.3,2/3+0.3) - - (2/3+0.3,2/3+0.3) - - (2/3+0.3,-2/3-0.3) - - (-2/3-0.1,-2/3-0.3) - - (-2/3-0.1,-2/3+0.1) - - (-2/3+0.1,-2/3+0.1) - - (-2/3+0.1,-2/3-0.2)  - -
(2/3+0.2,-2/3-0.2) - -(2/3+0.2,2/3+0.2) - - (-2/3-0.1,2/3+0.2)
- - (-2/3-0.1,2/3-0.1) - - (-2/3+0.1,2/3-0.1) - - (-2/3+0.1,2/3+0.1)- - (2/3+0.1,2/3+0.1)- - (2/3+0.1,-2/3-0.1) - - (2/3+0.1,-2/3-0.1) - - (-2/3+0.2,-2/3-0.1)  - - (-2/3+0.2,-2/3+0.2) - - (-2/3-0.2,-2/3+0.2)- - (-2/3-0.2,-2/3-0.4)- - (2/3+0.4,-2/3-0.4)- - (2/3+0.4,2/3+0.4)- - (0,2/3+0.4) - - (0,2)
;

\filldraw [gray] (-2,0) circle (1pt); 
\filldraw [gray] (-2/3,2/3) circle (1pt);    
\filldraw [gray] (2/3,2/3) circle (1pt);    
\filldraw [gray] (2/3,-2/3) circle (1pt);    
\filldraw [gray] (-2/3,-2/3) circle (1pt);


  \end{tikzpicture}
  }
  \caption{$p_1$ is in one component and $p_2,~p_2,~p_3$ are in the other}
  \label{figure 8}
\end{figure}

We intend to define $P$ and $R$ by grouping some terms in the sum originally used to define the representation in
terms of fork/noodle pairing. The pairing is defined as a certain sum (2.1) of terms corresponding to 
crossings between forks and noodles. We will choose some of the crossings to define $P$ and some other to 
define $R$. In order to do this we will need some preparations.

Let $T$ be the boundary of the square whose vertices are the puncture points. We denote the sides with 
$T_1,\ldots,T_4$, where $T_i$ connects $P_i$ with the next crossing (clockwise). We would like to work with 
a fork/noodle arrangement that has certain special properties. We need the pair (of a fork and a noodle) to be irreducible. 
We need the fork to be drawn in the standard way. We need the noodle to intersect $T$ transversally with 
minimum possible number of intersection points. 
And finally we need the tine of the fork
to intersect all segments at which the noodle intersects $T_4$ and $T_2$. While general position arguments show that we can take care of the first three conditions, there is no 
possibility of 
the fourth being satisfied without some further adjustments. Figure~9 shows an example.

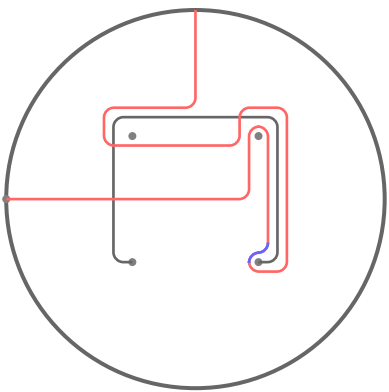
\begin{figure}[ht]
  \centering
  \resizebox{150pt}{!}{%
  \begin{tikzpicture}
  
\filldraw  [color=black!60, fill=white, very thick] (0,0) circle (2);

\filldraw [gray] (-2,0) circle (1pt); 
\filldraw [gray] (-2/3,2/3) circle (1pt);    
\filldraw [gray] (2/3,2/3) circle (1pt);    
\filldraw [gray] (2/3,-2/3) circle (1pt);    
\filldraw [gray] (-2/3,-2/3) circle (1pt);

\draw  [color=black!60, thick, rounded corners=3pt] (-2/3,-2/3) - - (-2/3-0.2,-2/3) - - (-2/3-0.2,2/3+0.2) - -(2/3+0.2,2/3+0.2) - -(2/3+0.2,-2/3) - -(2/3,-2/3);
\draw  [color=red!60, thick, rounded corners=3pt] (0, 2) - - (0,2/3+0.3) - - (-2/3-0.3,2/3+0.3) - -(-2/3-0.3,2/3-0.1) - -(2/3-0.2,2/3-0.1) - -(2/3-0.2,2/3+0.3) - - (2/3+0.3,2/3+0.3) - - (2/3+0.3,-2/3-0.1) - - (2/3-0.1,-2/3-0.1) - - (2/3-0.1,-2/3+0.1) - - (2/3+0.1,-2/3+0.1) - - (2/3+0.1,2/3+0.1)- - (2/3-0.1,2/3+0.1)- - (2/3-0.1,0) - - (-2,0);

\draw  [color=blue!60, thick,rounded corners=3pt] (2/3-0.1,-2/3) - -(2/3-0.1,-2/3+0.1) - - (2/3+0.1,-2/3+0.1) - - (2/3+0.1,-2/3+0.2);

  \end{tikzpicture}
  }
  \caption{The tine $T(F_1)a$ (shown as the black curve) does not intersect the blue segment at which the noodle $N$ (the union of blue and red segments) intersects $T_2$}
  \label{figure 9}
\end{figure}

 \ \
 
 \ \
 
However it is automatically
corrected if we increase the exponent $n$. The effect is just that we add a  number of turns around
 two pairs of punctures. They do not affect the three properties already dealt with and with 
sufficient increase of $n$ we obtain the fourth property. So we are interested in strings between puncture points which has transversal intersection with  $T$ (see Figure~10). 
\begin{figure}[ht]
  \centering
 \resizebox{150pt}{!}{%
  \begin{tikzpicture}

\filldraw[color=black!60, fill=white, very thick](0,0) circle (2);

\draw[color=black!60, dashed] (-2/3-0.4, 2/3) arc (180:90:0.4);
\draw[color=black!60, dashed] (-2/3, 2/3+0.4) - -  (2/3, 2/3+0.4);
\draw[color=black!60, dashed] (2/3, 2/3+0.4) arc (90:0:0.4);
\draw[color=black!60, dashed] (2/3+0.4, 2/3) - -  (2/3+0.4, -2/3);
\draw[color=black!60, fill=white, dashed] (2/3+0.4, -2/3) arc (0:-90:0.4);
\draw[color=black!60, dashed] (2/3, -2/3-0.4) - -  (-2/3, -2/3-0.4);
\draw[color=black!60, fill=white, dashed] (-2/3, -2/3-0.4) arc (-90:-180:0.4);
\draw[color=black!60, dashed] (-2/3-0.4, -2/3) - -  (-2/3-0.4, 2/3);

\draw[color=black!60, dashed] (-2/3+0.2, 2/3-0.4) arc (180:90:0.2);
\draw[color=black!60, dashed] (-2/3+0.4, 2/3-0.2) - -  (2/3-0.4, 2/3-0.2);
\draw[color=black!60, dashed] (2/3-0.4, 2/3-0.2) arc (90:0:0.2);
\draw[color=black!60, dashed] (2/3-0.2, 2/3-0.4) - -  (2/3-0.2, -2/3+0.4);
\draw[color=black!60, dashed] (2/3-0.2, -2/3+0.4) arc (0:-90:0.2);
\draw[color=black!60, dashed] (2/3-0.4, -2/3+0.2) - -  (-2/3+0.4, -2/3+0.2);
\draw[color=black!60, dashed] (-2/3+0.4, -2/3+0.2) arc (-90:-180:0.2);
\draw[color=black!60, dashed] (-2/3+0.2, -2/3+0.4) - -  (-2/3+0.2, 2/3-0.4);

\draw[color=red!60, thick] (-2/3-0.4, 0.2) - -  (-2/3+0.2, 0.2);
\draw[color=red!60, thick] (-2/3-0.4, -0.2) - -  (-2/3+0.2, -0.2);
\draw[color=red!60, thick] (2/3-0.2, 0.2) - -  (2/3+0.4, 0.2);
\draw[color=red!60, thick] (2/3-0.2, -0.2) - -  (2/3+0.4, -0.2);

\draw[color=red!60, thick] (-0.2, 2/3+0.4) - -  (-0.2, 2/3-0.2);
\draw[color=red!60, thick] (0.2, 2/3+0.4) - -  (0.2, 2/3-0.2);
\draw[color=red!60, thick] (-0.2, -2/3+0.2) - -  (-0.2, -2/3-0.4);
\draw[color=red!60, thick] (0.2, -2/3+0.2) - -  (0.2, -2/3-0.4);

\filldraw [gray] (-2,0) circle (1.5pt); 
\filldraw [gray] (-2/3-0.15,2/3+0.15) circle (1pt);    
\filldraw [gray] (2/3+0.15,2/3+0.15) circle (1pt);    
\filldraw [gray] (2/3+0.15,-2/3-0.15) circle (1pt);    
\filldraw [gray] (-2/3-0.15,-2/3-0.15) circle (1pt);

\filldraw [red] (-2/3,0) circle (0.5pt);
\filldraw [red] (2/3,0) circle (0.5pt);
\filldraw [red] (-2/3,0.1) circle (0.5pt);
\filldraw [red] (2/3,0.1) circle (0.5pt);
\filldraw [red] (-2/3,-0.1) circle (0.5pt);
\filldraw [red] (2/3,-0.1) circle (0.5pt);

\filldraw [red] (-0.1,2/3) circle (0.5pt);
\filldraw [red] (0,2/3) circle (0.5pt);
\filldraw [red] (0.1,2/3) circle (0.5pt);
\filldraw [red] (-0.1, -2/3) circle (0.5pt);
\filldraw [red] (0, -2/3) circle (0.5pt);
\filldraw [red] (0.1, -2/3) circle (0.5pt);

\draw[color=black!60, thick] (-2/3-0.15, 2/3+0.15) - - (2/3+0.15, 2/3+0.15);
\draw[color=black!60, thick] (2/3+0.15, 2/3+0.15) - - (2/3+0.15, -2/3-0.15);
\draw[color=black!60,  thick] (2/3+0.15, -2/3-0.15) - - (-2/3-0.15, -2/3-0.15);
\draw[color=black!60, thick] (-2/3-0.15, -2/3-0.15) - - (-2/3-0.15, 2/3+0.15);

\node[above] at (0, 2/3) {$T_1$};
\node[right] at (2/3,0) {$T_2$};
\node[below] at (0,-2/3-0.1) {$T_3$};
\node[left] at (-2/3, 0) {$T_4$};

\filldraw [blue] (-0.15-2/3,0.2) circle (1pt); 
\filldraw [blue] (-0.15-2/3,-0.2) circle (1pt); 
\filldraw [blue] (0.15+2/3,0.2) circle (1pt); 
\filldraw [blue] (0.15+2/3,-0.2) circle (1pt);
\filldraw [blue] (-0.2,2/3+0.15) circle (1pt); 
\filldraw [blue] (0.2,2/3+0.15) circle (1pt);
\filldraw [blue] (-0.2,-2/3-0.15) circle (1pt); 
\filldraw [blue] (0.2,-2/3-0.15) circle (1pt);


  \end{tikzpicture}
  }
  \caption{}
  \label{figure 10}
\end{figure}

Note that it is possible to be no such string between $p_2$ and $p_3$ or $p_3$ and $p_4$, but by our assumption ($p_1$ 
is in the first component) there is an odd number of strings between  $p_1$ and $p_2$ and an 
odd number between $p_1$ and $p_4$ (this guaranties that $P$ 
and $Q$ are not zero).

The pictures of $T\left(F_1\right)a^n$ and $T\left(F_3\right)a^n$ are as given in Figure~11. Therefore, they differ from each other by just one string and by the direction. The number of strings around $p_1$, $p_4$ and $p_2$, $p_3$ is $n$ for $T\left(F_1\right)a^n$ and $n-1$ for $T\left(F_3\right)a^n$.

\begin{figure}[ht]
  \centering
\resizebox{330pt}{!}{%
  \begin{tikzpicture}
  
\filldraw  [color=black!60, fill=white, very thick] (-2.5,0) circle (2);

\filldraw [gray] (-2-2.5,0) circle (1.5pt); 
\filldraw [gray] (-2/3-2.5,2/3) circle (1pt);    
\filldraw [gray] (2/3-2.5,2/3) circle (1pt);    
\filldraw [gray] (2/3-2.5,-2/3) circle (1pt);    
\filldraw [gray] (-2/3-2.5,-2/3) circle (1pt);

\draw  [color=black!60, thick, rounded corners=3pt] (-2/3-2.5,2/3) - - (-2/3-2.5+0.1,2/3) - - (-2/3-2.5+0.1,-2/3-0.1) - -(-2/3-2.5-0.1,-2/3-0.1) - -(-2/3-2.5-0.1,2/3+0.1) - -(-2/3-2.5+0.2,2/3+0.1) - -(-2/3-2.5+0.2,0.2);
\draw  [color=black!60, thick, rounded corners=3pt] (-2/3-2.5+0.2,-0.2) - - (-2/3-2.5+0.2,-2/3-0.2) - - (-2/3-2.5-0.2,-2/3-0.2) - - (-2/3-2.5-0.2,-0.2);
\draw  [color=black!60, thick, rounded corners=3pt] (-2/3-2.5-0.2,0.2) - - (-2/3-2.5-0.2,2/3+0.2) - - (-2/3-2.5+0.3,2/3+0.2) - - (-2/3-2.5+0.3,-2/3-0.3) - - (-2/3-2.5-0.3,-2/3-0.3) - - (-2/3-2.5-0.3,2/3+0.3) - - (2/3-2.5+0.3,2/3+0.3) - - (2/3-2.5+0.3,-2/3-0.3)- - (2/3-2.5-0.3,-2/3-0.3) - - (2/3-2.5-0.3,2/3+0.2) - - (2/3-2.5+0.2,2/3+0.2)- - (2/3-2.5+0.2,0.2);
\draw  [color=black!60, thick, rounded corners=3pt] (2/3-2.5+0.2,-0.2) - - (2/3-2.5+0.2,-2/3-0.2) - - (2/3-2.5-0.2,-2/3-0.2) - - (2/3-2.5-0.2,-0.2);
\draw  [color=black!60, thick, rounded corners=3pt] (2/3-2.5-0.2,0.2) - - (2/3-2.5-0.2,2/3+0.1) - - (2/3-2.5+0.1,2/3+0.1) - - (2/3-2.5+0.1,-2/3-0.1) - - (2/3-2.5-0.1,-2/3-0.1)- - (2/3-2.5-0.1,2/3) - - (2/3-2.5,2/3);
\draw  [color=black!60, thick, ->] (-2.5-0.2,2/3+0.3) - - (-2.5+0.2,2/3+0.3);

\filldraw [gray] (-2/3-2.5-0.2,0) circle (0.5pt);
\filldraw [gray] (-2/3-2.5-0.2,0.1) circle (0.5pt);
\filldraw [gray] (-2/3-2.5-0.2,-0.1) circle (0.5pt);
\filldraw [gray] (-2/3-2.5+0.2,0) circle (0.5pt);
\filldraw [gray] (-2/3-2.5+0.2,0.1) circle (0.5pt);
\filldraw [gray] (-2/3-2.5+0.2,-0.1) circle (0.5pt);

\filldraw [gray] (2/3-2.5-0.2,0) circle (0.5pt);
\filldraw [gray] (2/3-2.5-0.2,0.1) circle (0.5pt);
\filldraw [gray] (2/3-2.5-0.2,-0.1) circle (0.5pt);
\filldraw [gray] (2/3-2.5+0.2,0) circle (0.5pt);
\filldraw [gray] (2/3-2.5+0.2,0.1) circle (0.5pt);
\filldraw [gray] (2/3-2.5+0.2,-0.1) circle (0.5pt);

\filldraw  [color=black!60, fill=white, very thick] (2.5,0) circle (2);

\filldraw [gray] (-2+2.5,0) circle (1.5pt); 
\filldraw [gray] (-2/3+2.5,2/3) circle (1pt);    
\filldraw [gray] (2/3+2.5,2/3) circle (1pt);    
\filldraw [gray] (2/3+2.5,-2/3) circle (1pt);    
\filldraw [gray] (-2/3+2.5,-2/3) circle (1pt);

\draw  [color=black!60, thick, rounded corners=3pt] (-2/3+2.5,-2/3) - - (-2/3+2.5-0.1,-2/3) - -(-2/3+2.5-0.1,2/3+0.1) - -(-2/3+2.5+0.2,2/3+0.1) - -(-2/3+2.5+0.2,0.2);
\draw  [color=black!60, thick, rounded corners=3pt] (-2/3+2.5+0.2,-0.2) - - (-2/3+2.5+0.2,-2/3-0.2) - - (-2/3+2.5-0.2,-2/3-0.2) - - (-2/3+2.5-0.2,-0.2);
\draw  [color=black!60, thick, rounded corners=3pt] (-2/3+2.5-0.2,0.2) - - (-2/3+2.5-0.2,2/3+0.2) - - (-2/3+2.5+0.3,2/3+0.2) - - (-2/3+2.5+0.3,-2/3-0.3) - - (-2/3+2.5-0.3,-2/3-0.3) - - (-2/3+2.5-0.3,2/3+0.3) - - (2/3+2.5+0.3,2/3+0.3) - - (2/3+2.5+0.3,-2/3-0.3)- - (2/3+2.5-0.3,-2/3-0.3) - - (2/3+2.5-0.3,2/3+0.2) - - (2/3+2.5+0.2,2/3+0.2)- - (2/3+2.5+0.2,0.2);
\draw  [color=black!60, thick, rounded corners=3pt] (2/3+2.5+0.2,-0.2) - - (2/3+2.5+0.2,-2/3-0.2) - - (2/3+2.5-0.2,-2/3-0.2) - - (2/3+2.5-0.2,-0.2);
\draw  [color=black!60, thick, rounded corners=3pt] (2/3+2.5-0.2,0.2) - - (2/3+2.5-0.2,2/3+0.1) - - (2/3+2.5+0.1,2/3+0.1) - - (2/3+2.5+0.1,-2/3) - - (2/3+2.5,-2/3);

\draw  [color=black!60, thick, ->] (2.5+0.2,2/3+0.3) - - (2.5-0.2,2/3+0.3);

\filldraw [gray] (-2/3+2.5-0.2,0) circle (0.5pt);
\filldraw [gray] (-2/3+2.5-0.2,0.1) circle (0.5pt);
\filldraw [gray] (-2/3+2.5-0.2,-0.1) circle (0.5pt);
\filldraw [gray] (-2/3+2.5+0.2,0) circle (0.5pt);
\filldraw [gray] (-2/3+2.5+0.2,0.1) circle (0.5pt);
\filldraw [gray] (-2/3+2.5+0.2,-0.1) circle (0.5pt);

\filldraw [gray] (2/3+2.5-0.2,0) circle (0.5pt);
\filldraw [gray] (2/3+2.5-0.2,0.1) circle (0.5pt);
\filldraw [gray] (2/3+2.5-0.2,-0.1) circle (0.5pt);
\filldraw [gray] (2/3+2.5+0.2,0) circle (0.5pt);
\filldraw [gray] (2/3+2.5+0.2,0.1) circle (0.5pt);
\filldraw [gray] (2/3+2.5+0.2,-0.1) circle (0.5pt);

  \end{tikzpicture}
}
  \caption{}
  \label{figure 11}
\end{figure}

If we take curves $T\left(F_1\right)a$ and $N_1{\sigma }^{-1}$ in the same $D_4$ and assume that their intersection is transversal, then it is possible that $T\left(F_1\right)a$ does not intersect all strings between $p_1$ and $p_4$ or $p_2$ and  $p_3$. For example see Figure~9.

\begin{figure}[ht]
  \centering
  \resizebox{330pt}{!}{%
  \begin{tikzpicture}
  
\filldraw  [color=black!60, fill=white, very thick] (-2.5,0) circle (2);

\filldraw [gray] (-2-2.5,0) circle (1.5pt); 
\filldraw [gray] (-2/3-2.5,2/3) circle (1pt);    
\filldraw [gray] (2/3-2.5,2/3) circle (1pt);    
\filldraw [gray] (2/3-2.5,-2/3) circle (1pt);    
\filldraw [gray] (-2/3-2.5,-2/3) circle (1pt);

\draw  [color=black!60, thick, rounded corners=3pt] (2/3-2.5+0.1,2/3+0.3)- - (2/3-2.5+0.1,-2/3-0.1) - - (2/3-2.5-0.1,-2/3-0.1)- - (2/3-2.5-0.1,2/3) - - (2/3-2.5,2/3);
\draw  [color=black!60, thick, rounded corners=3pt] (-2/3-2.5-0.1,2/3+0.3)- - (-2/3-2.5-0.1,-2/3-0.1) - - (-2/3-2.5+0.1,-2/3-0.1)- - (-2/3-2.5+0.1,2/3) - - (-2/3-2.5,2/3);

\draw  [color=black!60, thick, ->] (-2/3-2.5-0.1,-0.2)- - (-2/3-2.5-0.1,0.2);
\draw  [color=black!60, thick, ->] (2/3-2.5+0.1,0.2)- - (2/3-2.5+0.1,-0.2);

\draw  [color=red!60, thick] (-2/3-2.5-0.5,0.2) - - (-2/3-2.5+0.5,0.2);
\draw  [color=red!60, thick] (-2/3-2.5-0.5,-0.2) - - (-2/3-2.5+0.5,-0.2);
\draw  [color=red!60, thick] (2/3-2.5-0.5,0.2) - - (2/3-2.5+0.5,0.2);
\draw  [color=red!60, thick] (2/3-2.5-0.5,-0.2) - - (2/3-2.5+0.5,-0.2);

\filldraw [red] (-2/3-2.5,0) circle (0.5pt);
\filldraw [red] (-2/3-2.5,0.1) circle (0.5pt);
\filldraw [red] (-2/3-2.5,-0.1) circle (0.5pt);
\filldraw [red] (2/3-2.5,0) circle (0.5pt);
\filldraw [red] (2/3-2.5,0.1) circle (0.5pt);
\filldraw [red] (2/3-2.5,-0.1) circle (0.5pt);

\filldraw  [color=black!60, fill=white, very thick] (2.5,0) circle (2);

\filldraw [gray] (-2+2.5,0) circle (1.5pt); 
\filldraw [gray] (-2/3+2.5,2/3) circle (1pt);    
\filldraw [gray] (2/3+2.5,2/3) circle (1pt);    
\filldraw [gray] (2/3+2.5,-2/3) circle (1pt);    
\filldraw [gray] (-2/3+2.5,-2/3) circle (1pt);

\draw  [color=black!60, thick, rounded corners=3pt] (-2/3+2.5, -2/3) - - (-2/3+2.5-0.1,-2/3) - - (-2/3+2.5-0.1,2/3+0.3);
\draw  [color=black!60, thick, rounded corners=3pt](2/3+2.5+0.1,2/3+0.3) - - (2/3+2.5+0.1,-2/3) - - (2/3+2.5,-2/3);

\draw  [color=black!60, thick, ->] (-2/3+2.5-0.1,0.2)- - (-2/3+2.5-0.1,-0.2);
\draw  [color=black!60, thick, ->] (2/3+2.5+0.1,-0.2)- - (2/3+2.5+0.1,0.2);

\draw  [color=red!60, thick] (-2/3+2.5-0.5,0.2) - - (-2/3+2.5+0.5,0.2);
\draw  [color=red!60, thick] (-2/3+2.5-0.5,-0.2) - - (-2/3+2.5+0.5,-0.2);
\draw  [color=red!60, thick] (2/3+2.5-0.5,0.2) - - (2/3+2.5+0.5,0.2);
\draw  [color=red!60, thick] (2/3+2.5-0.5,-0.2) - - (2/3+2.5+0.5,-0.2);

\filldraw [red] (-2/3+2.5,0) circle (0.5pt);
\filldraw [red] (-2/3+2.5,0.1) circle (0.5pt);
\filldraw [red] (-2/3+2.5,-0.1) circle (0.5pt);
\filldraw [red] (2/3+2.5,0) circle (0.5pt);
\filldraw [red] (2/3+2.5,0.1) circle (0.5pt);
\filldraw [red] (2/3+2.5,-0.1) circle (0.5pt);

\draw[color=black!60, dashed](-2/3-2.5,0) ellipse (0.5 and 1);
\draw[color=black!60, dashed](2/3-2.5,0) ellipse (0.5 and 1);

\draw[color=black!60, dashed](-2/3+2.5,0) ellipse (0.5 and 1);
\draw[color=black!60, dashed](2/3+2.5,0) ellipse (0.5 and 1);

\filldraw [blue] (-0.1-2/3-2.5,0.2) circle (1pt); 
\filldraw [blue] (-0.1-2/3-2.5,-0.2) circle (1pt); 
\filldraw [blue] (0.1+2/3-2.5,0.2) circle (1pt); 
\filldraw [blue] (0.1+2/3-2.5,-0.2) circle (1pt);

\filldraw [blue] (0.1-2/3-2.5,0.2) circle (1pt); 
\filldraw [blue] (0.1-2/3-2.5,-0.2) circle (1pt); 
\filldraw [blue] (-0.1+2/3-2.5,0.2) circle (1pt); 
\filldraw [blue] (-0.1+2/3-2.5,-0.2) circle (1pt);

\node[below] at (-2/3-2.5,-2/3-0.3) {$U_1$};
\node[below] at (2/3-2.5,-2/3-0.3) {$U_2$};
\node[below] at (-2.5,-2-0.3) {$a$};

\filldraw [blue] (-0.1-2/3+2.5,0.2) circle (1pt); 
\filldraw [blue] (-0.1-2/3+2.5,-0.2) circle (1pt); 
\filldraw [blue] (0.1+2/3+2.5,0.2) circle (1pt); 
\filldraw [blue] (0.1+2/3+2.5,-0.2) circle (1pt);

\node[below] at (-2/3+2.5,-2/3-0.3) {$U_1$};
\node[below] at (2/3+2.5,-2/3-0.3) {$U_2$};
\node[below] at (2.5,-2-0.3) {$b$};

  \end{tikzpicture}
  }
  \caption{}
  \label{figure 12}
\end{figure}

In this case we must take more numbers of $a$'s and finally we will obtain the curves $T\left(F_1\right)a^n$ and $N_1{\sigma }^{-1}$ such that we can find neighborhoods $U_1$ and $U_2$ of  $T_{4}$ and $T_2$ respectively, with the following picture, illustrated in Figure~12a.

In this case for Figure~12a  the polynomial corresponding to intersections 
inside $U_1$ and  $U_2$ can be written as $P(t,t^{-1})(1-t^{-1})$ and 
$R(t,t^{-1})(1-t)$ respectively. 
Let $Q\left(t,t^{-1}\right)={\rho }_{11}\left(a^n\sigma \right)-P\left(t,t^{-1}\right)\left(1-t^{-1}\right)-R(t,t^{-1})(1-t)$, then we have
\[{\rho }_{11}\left(a^n\sigma \right)=P\left(t,t^{-1}\right)\left(1-t^{-1}\right)+Q\left(t,t^{-1}\right)+R\left(t,t^{-1}\right)(1-t).\] 
On the other hand if we look at Figure~12b and keep in mind that directions of  
$T\left(F_1\right)a^n$ and $T\left(F_3\right)a^n$ are different we can say that 
\[{\ \rho }_{31}\left(a^n\sigma \right)=-P\left(t,t^{-1}\right)-Q\left(t,t^{-1}\right)-R\left(t,t^{-1}\right).\]

 After that if we multiply the braid $a^n\sigma $ by $a$ on the left side then we obtain 
the following picture, illustrated  in Figure~13.

\begin{figure}[ht]
  \centering
    \resizebox{330pt}{!}{%
  \begin{tikzpicture}
  
\filldraw  [color=black!60, fill=white, very thick] (-2.5,0) circle (2);

\filldraw [gray] (-2-2.5,0) circle (1.5pt); 
\filldraw [gray] (-2/3-2.5,2/3) circle (1pt);    
\filldraw [gray] (2/3-2.5,2/3) circle (1pt);    
\filldraw [gray] (2/3-2.5,-2/3) circle (1pt);    
\filldraw [gray] (-2/3-2.5,-2/3) circle (1pt);

\draw  [color=black!60, thick, rounded corners=3pt] (2/3-2.5+0.2,2/3+0.25)- - (2/3-2.5+0.2,-2/3-0.1) - - (2/3-2.5-0.1,-2/3-0.1)- - (2/3-2.5-0.1,2/3+0.1) - - (2/3-2.5+0.1,2/3+0.1)- - (2/3-2.5+0.1,-2/3)- - (2/3-2.5,-2/3);
\draw  [color=black!60, thick, rounded corners=3pt] (-2/3-2.5-0.2,2/3+0.25)- - (-2/3-2.5-0.2,-2/3-0.1) - - (-2/3-2.5+0.1,-2/3-0.1)- - (-2/3-2.5+0.1,2/3+0.1) - - (-2/3-2.5-0.1,2/3+0.1)- - (-2/3-2.5-0.1,-2/3)- - (-2/3-2.5,-2/3);

\draw  [color=black!60, thick, ->] (-2/3-2.5-0.2,-0.2)- - (-2/3-2.5-0.2,0.2);
\draw  [color=black!60, thick, ->] (2/3-2.5+0.2,0.2)- - (2/3-2.5+0.2,-0.2);

\draw  [color=red!60, thick] (-2/3-2.5-0.5,0.2) - - (-2/3-2.5+0.5,0.2);
\draw  [color=red!60, thick] (-2/3-2.5-0.5,-0.2) - - (-2/3-2.5+0.5,-0.2);
\draw  [color=red!60, thick] (2/3-2.5-0.5,0.2) - - (2/3-2.5+0.5,0.2);
\draw  [color=red!60, thick] (2/3-2.5-0.5,-0.2) - - (2/3-2.5+0.5,-0.2);

\filldraw [red] (-2/3-2.5,0) circle (0.5pt);
\filldraw [red] (-2/3-2.5,0.1) circle (0.5pt);
\filldraw [red] (-2/3-2.5,-0.1) circle (0.5pt);
\filldraw [red] (2/3-2.5,0) circle (0.5pt);
\filldraw [red] (2/3-2.5,0.1) circle (0.5pt);
\filldraw [red] (2/3-2.5,-0.1) circle (0.5pt);

\draw[color=black!60, dashed](-2/3-2.5,0) ellipse (0.5 and 1);
\draw[color=black!60, dashed](2/3-2.5,0) ellipse (0.5 and 1);

\filldraw  [color=black!60, fill=white, very thick] (2.5,0) circle (2);

\filldraw [gray] (-2+2.5,0) circle (1.5pt); 
\filldraw [gray] (-2/3+2.5,2/3) circle (1pt);    
\filldraw [gray] (2/3+2.5,2/3) circle (1pt);    
\filldraw [gray] (2/3+2.5,-2/3) circle (1pt);    
\filldraw [gray] (-2/3+2.5,-2/3) circle (1pt);

\draw  [color=black!60, thick, rounded corners=3pt] (-2/3+2.5, 2/3) - -(-2/3+2.5+0.1, 2/3)- -(-2/3+2.5+0.1, -2/3-0.1) - -(-2/3+2.5-0.1, -2/3-0.1)- - (-2/3+2.5-0.1,2/3+0.25);
\draw  [color=black!60, thick, rounded corners=3pt](2/3+2.5,2/3) - - (2/3+2.5-0.1,2/3) - -(2/3+2.5-0.1,-2/3-0.1) - - (2/3+2.5+0.1,-2/3-0.1) - - (2/3+2.5+0.1,2/3+0.25);

\draw  [color=black!60, thick, ->] (-2/3+2.5-0.1,0.2)- - (-2/3+2.5-0.1,-0.2);
\draw  [color=black!60, thick, ->] (2/3+2.5+0.1,-0.2)- - (2/3+2.5+0.1,0.2);

\draw  [color=red!60, thick] (-2/3+2.5-0.5,0.2) - - (-2/3+2.5+0.5,0.2);
\draw  [color=red!60, thick] (-2/3+2.5-0.5,-0.2) - - (-2/3+2.5+0.5,-0.2);
\draw  [color=red!60, thick] (2/3+2.5-0.5,0.2) - - (2/3+2.5+0.5,0.2);
\draw  [color=red!60, thick] (2/3+2.5-0.5,-0.2) - - (2/3+2.5+0.5,-0.2);

\filldraw [red] (-2/3+2.5,0) circle (0.5pt);
\filldraw [red] (-2/3+2.5,0.1) circle (0.5pt);
\filldraw [red] (-2/3+2.5,-0.1) circle (0.5pt);
\filldraw [red] (2/3+2.5,0) circle (0.5pt);
\filldraw [red] (2/3+2.5,0.1) circle (0.5pt);
\filldraw [red] (2/3+2.5,-0.1) circle (0.5pt);

\draw[color=black!60, dashed](-2/3+2.5,0) ellipse (0.5 and 1);
\draw[color=black!60, dashed](2/3+2.5,0) ellipse (0.5 and 1);

\filldraw [blue] (-0.2-2/3-2.5,0.2) circle (1pt); 
\filldraw [blue] (-0.2-2/3-2.5,-0.2) circle (1pt);
\filldraw [blue] (-0.1-2/3-2.5,0.2) circle (1pt); 
\filldraw [blue] (-0.1-2/3-2.5,-0.2) circle (1pt); 

\filldraw [blue] (0.1+2/3-2.5,0.2) circle (1pt); 
\filldraw [blue] (0.1+2/3-2.5,-0.2) circle (1pt);
\filldraw [blue] (0.2+2/3-2.5,0.2) circle (1pt); 
\filldraw [blue] (0.2+2/3-2.5,-0.2) circle (1pt);
 
\filldraw [blue] (0.1-2/3-2.5,0.2) circle (1pt); 
\filldraw [blue] (0.1-2/3-2.5,-0.2) circle (1pt); 
\filldraw [blue] (-0.1+2/3-2.5,0.2) circle (1pt); 
\filldraw [blue] (-0.1+2/3-2.5,-0.2) circle (1pt);

\node[below] at (-2/3-2.5,-2/3-0.3) {$U_1$};
\node[below] at (2/3-2.5,-2/3-0.3) {$U_2$};
\node[below] at (-2.5,-2-0.3) {$a$};

\filldraw [blue] (0.1-2/3+2.5,0.2) circle (1pt); 
\filldraw [blue] (0.1-2/3+2.5,-0.2) circle (1pt);
\filldraw [blue] (-0.1-2/3+2.5,0.2) circle (1pt); 
\filldraw [blue] (-0.1-2/3+2.5,-0.2) circle (1pt); 
\filldraw [blue] (0.1+2/3+2.5,0.2) circle (1pt); 
\filldraw [blue] (0.1+2/3+2.5,-0.2) circle (1pt);
\filldraw [blue] (-0.1+2/3+2.5,0.2) circle (1pt); 
\filldraw [blue] (-0.1+2/3+2.5,-0.2) circle (1pt);

\node[below] at (-2/3+2.5,-2/3-0.3) {$U_1$};
\node[below] at (2/3+2.5,-2/3-0.3) {$U_2$};
\node[below] at (2.5,-2-0.3) {$b$};

  \end{tikzpicture}
  }
  \caption{}
  \label{figure 13}
\end{figure}

Therefore we will have
\[{\rho }_{11}\left(a^{n+1}\sigma \right)=P\left(t,t^{-1}\right)\left(1-t^{-1}+t^{-2}\right)+\] 
\[Q\left(t,t^{-1}\right)+R\left(t,t^{-1}\right)\left(1-t^1+t^2\right),\] 
\[{\rho }_{31}\left(a^{n+1}\sigma \right)=-P\left(t,t^{-1}\right)\left(1-t^{-1}\right)-Q\left(t,t^{-1}\right)-R\left(t,t^{-1}\right)\left(1-t^1\right).\] 
Note that the same argument it sufficient to complete the proof.  
\end{proof}

\textbf{Example 1.} Let $\sigma =b^{-1}a^{-1}b$, then for $n=2$ we have
\[{\rho }_{11}\left(a^2b^{-1}a^{-1}b\right)=-t^{-3}+t^{-2}-1+2t-t^2-t^3+2t^4-t^5=\] 
\[t^{-2}\left(1-t^{-1}\right)+\left(-1+t-t^3+t^4\right)\left(1-t\right),\] 
\[{\rho }_{31}\left(a^2b^{-1}a^{-1}b\right)=-t^{-2}+1-t+t^3-t^4=-t^{-2}-\left(-1+t-t^3+t^4\right).\] 
See Figure~14. Let $m=3$, then we can see that

\ \

\ \

 \ \

\begin{figure}[ht]
  \centering
  \resizebox{150pt}{!}{%
  \begin{tikzpicture}

\filldraw[color=black!60, fill=white, very thick](0,0) circle (2);

\filldraw [gray] (-2,0) circle (1.5pt); 
\filldraw [gray] (-2/3,2/3) circle (1pt);    
\filldraw [gray] (2/3,2/3) circle (1pt);    
\filldraw [gray] (2/3,-2/3) circle (1pt);    
\filldraw [gray] (-2/3,-2/3) circle (1pt);

\draw[color=black!60, thick, rounded corners=3pt](-2/3,2/3) - - (-2/3
+0.1,2/3) - - (-2/3
+0.1,-2/3-0.1) - - (-2/3
-0.1,-2/3-0.1)  - - (-2/3
-0.1,2/3+0.1)- - (2/3
+0.1,2/3+0.1)- - (2/3
+0.1,-2/3-0.1)- - (2/3
-0.1,-2/3-0.1)- - (2/3
-0.1,2/3)- - (2/3,2/3);

\draw[color=red!60, thick, rounded corners=3pt](-2,0) - - (-2/3-0.6,0)- - (-2/3-0.6,-2/3-0.6)- - (2/3+0.3,-2/3-0.6) - - (2/3+0.3,-2/3+0.2) - - (2/3-0.3,-2/3+0.2) - - (2/3-0.3,-2/3-0.3) - - (-2/3-0.3,-2/3-0.3) - - (-2/3-0.3,2/3+0.2) - - (2/3+0.2,2/3+0.2) - - (2/3+0.2,2/3-0.1) - - (-2/3-0.2,2/3-0.1) - - (-2/3-0.2,-2/3-0.2) - - (-2/3+0.2,-2/3-0.2) - - (-2/3+0.2,2/3-0.2) - - (2/3+0.3,2/3-0.2) - - (2/3+0.3,2/3+0.3) - - (-2/3-0.4,2/3+0.3) - - (-2/3-0.4,-2/3-0.4)- - (2/3-0.2,-2/3-0.4)- - (2/3-0.2,-2/3+0.1)- - (2/3+0.2,-2/3+0.1)- - (2/3+0.2,-2/3-0.5)- - (-2/3-0.5,-2/3-0.5)- - (-2/3-0.5,2/3+0.4)- - (0, 2/3+0.4)- - (0,2);

\filldraw [blue] (-2/3-0.1,2/3-0.1) circle (1pt);
\filldraw [blue] (2/3+0.1,2/3-0.1) circle (1pt);
\filldraw [blue] (2/3+0.1,2/3-0.2) circle (1pt);
\filldraw [blue] (2/3+0.1,-2/3+0.2) circle (1pt);
\filldraw [blue] (2/3+0.1,-2/3+0.1) circle (1pt);



  \end{tikzpicture}
  }
 \caption{The polynomial corresponding to the marked intersection point on the left side
 is $P(t,t^{-1})=t^2$ and the polynomial corresponding to the marked intersection points on the right
 side is $R(t,t^{-1})=-1+t-t^3+t^4$}
 \label{figure 14}
\end{figure}

\[{\rho }_{11}\left(a^5b^{-1}a^{-1}b\right)=t^{-2}\left(1-t^{-1}+t^{-2}-t^{-3}+t^{-4}\right)+\] 
\[\left(-1+t-t^3+t^4\right)\left(1-t+t^2-t^3+t^4\right)=\] 
\[t^{-6}-t^{-5}+t^{-4}-t^{-3}+t^{-2}-\] 
\[1+2t-2t^2+t^3-t^5+2t^6-2t^7+t^8,\] 
\[{\rho }_{31}\left(a^5b^{-1}a^{-1}b\right)=-t^{-2}\left(1-t^{-1}+t^{-2}-t^{-3}\right)-\] 
\[\left(-1+t-t^3+t^4\right)\left(1-t+t^2-t^3\right)=\] 
\[t^{-5}-t^{-4}+t^{-3}-t^{-2}+\] 
\[1-2t+2t^2-t^3-t^4+2t^5-2t^6+t^7\] 

\textbf{Example 2.} Let $\sigma =ba^{-2}b^{-1}$, then for $n=2$ we have
\[{\rho }_{11}\left(a^2ba^{-2}b^{-1}\right)=t^{-6}-2t^{-5}+t^{-4}+t^{-3}-t^{-2}+t^{-1}=\] 
\[\left(-t^{-5}+t^{-4}-t^{-2}\right)\left(1-t^{-1}\right)+t^{-1},\] 
\[{\rho }_{31}\left(a^2ba^{-2}b^{-1}\right)=t^{-5}-t^{-4}+t^{-2}-t^{-1}=-\left(-t^{-5}+t^{-4}-t^{-2}\right)-t^{-1}.\] 
See Figure~15 Let $m=4$ then we can see that

\begin{figure}[ht]
  \centering
   \resizebox{150pt}{!}{%
  \begin{tikzpicture}

\filldraw[color=black!60, fill=white, thick](0,0) circle (2);

\filldraw [blue] (-2,0) circle (1pt); 
\filldraw [blue] (-2/3,2/3) circle (1pt);    
\filldraw [blue] (2/3,2/3) circle (1pt);    
\filldraw [blue] (2/3,-2/3) circle (1pt);    
\filldraw [blue] (-2/3,-2/3) circle (1pt);

\draw[color=black!60, thick, rounded corners=3pt](-2/3,2/3) - - (-2/3
+0.1,2/3) - - (-2/3
+0.1,-2/3-0.1) - - (-2/3
-0.1,-2/3-0.1)  - - (-2/3
-0.1,2/3+0.2)- - (2/3
+0.1,2/3+0.2)- - (2/3
+0.1,-2/3-0.1)- - (2/3
-0.1,-2/3-0.1)- - (2/3
-0.1,2/3)- - (2/3,2/3);

\draw[color=red!60, thick, rounded corners=3pt](-2,0) - - (-2/3-0.4,0) - - (-2/3-0.4,2/3+0.5) - - (2/3+0.5,2/3+0.5)- - (2/3+0.5,-2/3-0.5)- - (-2/3-0.2,-2/3-0.5) - - (-2/3-0.2,-2/3+0.1) - - (-2/3+0.2,-2/3+0.1) - - (-2/3+0.2,-2/3-0.3) - - (2/3+0.3,-2/3-0.3) - - (2/3+0.3,2/3+0.3)- - (-2/3-0.2,2/3+0.3)- - (-2/3-0.2,2/3-0.1) - - (-2/3+0.2,2/3-0.1) - - (-2/3+0.2,2/3+0.1) - - (2/3+0.2,2/3+0.1) - - (2/3+0.2,-2/3-0.2)- - (-2/3+0.3,-2/3-0.2)- - (-2/3+0.3,-2/3+0.2) - - (-2/3-0.3,-2/3+0.2)- - (-2/3-0.3,-2/3-0.6)- - (2/3+0.6,-2/3-0.6)- - (2/3+0.6,2/3+0.6)- - (0,2/3+0.6)- - (0,2);

\filldraw [blue] (-2/3-0.1,2/3-0.1) circle (1pt);
\filldraw [blue] (-2/3-0.1,-2/3+0.2) circle (1pt);
\filldraw [blue] (-2/3-0.1,-2/3+0.1) circle (1pt);

\filldraw [blue] (2/3+0.1,2/3+0.1) circle (1pt);

  \end{tikzpicture}
  }
  \caption{The polynomial corresponding to the three marked intersection points on the left side is $P(t,t^{-1})=-t^{-5}+t^{-4}-t^{-2}$ and the polynomial corresponding to the single marked intersection point on the right side is the monomial $Q(t,t^{-1})=t^{-1}$}
  \label{figure 15}
\end{figure}
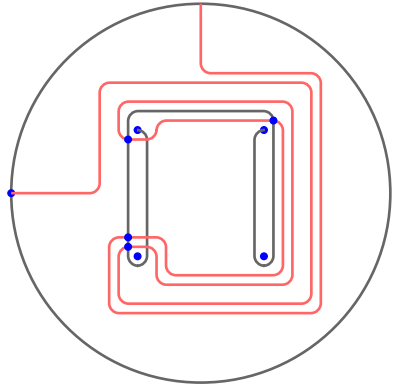

\[{\rho }_{11}\left(a^6ba^{-2}b^{-1}\right)=\left(-t^{-5}+t^{-4}-t^{-2}\right)\left(1-t^{-1}+t^{-2}-t^{-3}+t^{-4}-t^{-5}\right)+t^{-1}=\] 
\[t^{-10}-2t^{-9}+2t^{-8}-t^{-7}+t^{-6}-t^{-5}+t^{-3}-t^{-2}+t^{-1},\] 
\[{\rho }_{13}\left(a^6ba^{-2}b^{-1}\right)=\left(-t^{-5}+t^{-4}-t^{-2}\right)\left(1-t^{-1}+t^{-2}-t^{-3}+t^{-4}\right)+t^{-1}=\] 
\[t^{-9}-2t^{-8}+2t^{-7}-t^{-6}+t^{-5}-t^{-3}+t^{-2}-t^{-1}.\] 

\textbf{Example 3.} Let $\sigma =ab^2ab^{-1}$, then for $n=2$ we have
\[{\rho }_{11}\left(a^2ab^2ab^{-1}\right)=-t^{-6}+2t^{-5}-t^{-4}+t^{-3}+3t^{-2}-3t^{-1}+2-t-t^2+2t^3-t^4=\] 
\[\left(t^{-5}-t^{-4}+t^{-2}-2t^{-1}+1\right)\left(1-t^{-1}\right)\] 
\[\left(-t^{-4}+t^{-3}-2t^{-1}+2-t\right)+\left(1-t^2+t^3\right)\left(1-t\right),\] 
\[{\rho }_{11}\left(a^2ab^2ab^{-1}\right)=-t^{-5}+2t^{-4}-t^{-3}-t^{-2}+4t^{-1}-4+t+t^2-t^3=\] 
\[-\left(t^{-5}-t^{-4}+t^{-2}-2t^{-1}+1\right)-\] 
\[\left(-t^{-4}+t^{-3}-2t^{-1}+2-t\right)-\left(1-t^2+t^3\right).\] 
See Figure~16 Let $m=2$ then we can see that

\begin{figure}[ht]
  \centering
   \resizebox{150pt}{!}{%
  \begin{tikzpicture}

\filldraw[color=black!60, fill=white, very thick](4.2,0) circle (2);

\filldraw [gray] (-2+4.2,0) circle (1.5pt); 
\filldraw [gray] (-2/3+4.2,2/3) circle (1pt);    
\filldraw [gray] (2/3+4.2,2/3) circle (1pt);    
\filldraw [gray] (2/3+4.2,-2/3) circle (1pt);    
\filldraw [gray] (-2/3+4.2,-2/3) circle (1pt);

\draw[color=black!60, thick, rounded corners=3pt](-2/3+4.2,2/3) - - (-2/3+4.2
+0.1,2/3) - - (-2/3+4.2
+0.1,-2/3-0.1) - - (-2/3+4.2
-0.1,-2/3-0.1)  - - (-2/3+4.2
-0.1,2/3+0.2)- - (2/3+4.2
+0.1,2/3+0.2)- - (2/3+4.2
+0.1,-2/3-0.1)- - (2/3+4.2
-0.1,-2/3-0.1)- - (2/3+4.2
-0.1,2/3)- - (2/3+4.2,2/3);

\draw[color=red!60, thick, rounded corners=3pt](-2+4.2,0) - - (-2/3+4.2-0.75,0) - - (-2/3+4.2-0.75,2/3+0.4)- - (-2/3+4.2+0.3,2/3+0.4)- - (-2/3+4.2+0.3,2/3-0.2)- - (-2/3+4.2-0.35,2/3-0.2) - - (-2/3+4.2-0.35,-2/3-0.35)- - (2/3+4.2+0.2,-2/3-0.35)- - (2/3+4.2+0.2,-2/3+0.1)- - (2/3+4.2-0.2,-2/3+0.1)- - (2/3+4.2-0.2,-2/3-0.15)- - (-2/3+4.2-0.15,-2/3-0.15)- - (-2/3+4.2-0.15,-2/3+0.2)- - (2/3+4.2-0.25,-2/3+0.2)- - (2/3+4.2-0.25,2/3+0.35)- - (2/3+4.2+0.45,2/3+0.35)- - (2/3+4.2+0.45,-2/3-0.6)- - (-2/3+4.2-0.6,-2/3-0.6)- - (-2/3+4.2-0.6,2/3+0.25)- - (-2/3+4.2+0.15,2/3+0.25)- - (-2/3+4.2+0.15,2/3-0.05)- - (-2/3+4.2-0.5,2/3-0.05)- - (-2/3+4.2-0.5,-2/3-0.5) - - (2/3+4.2+0.35,-2/3-0.5)- - (2/3+4.2+0.35,2/3+0.25)- - (2/3+4.2-0.15,2/3+0.25)- - (2/3+4.2-0.15,2/3-0.05)- - (2/3+4.2+0.3,2/3-0.05)- - (2/3+4.2+0.3,-2/3-0.45)- - (-2/3+4.2-0.45,-2/3-0.45)- - (-2/3+4.2-0.45,2/3-0.1)- - (-2/3+4.2+0.2,2/3-0.1)- - (-2/3+4.2+0.2,2/3+0.3)- - (-2/3+4.2-0.65,2/3+0.3)- - (-2/3+4.2-0.65,-2/3-0.65)- - (2/3+4.2+0.5,-2/3-0.65)- - (2/3+4.2+0.5,2/3+0.4)- - (2/3+4.2-0.3,2/3+0.4)- - (2/3+4.2-0.3,-2/3+0.25)- - (-2/3+4.2-0.2,-2/3+0.25)- - (-2/3+4.2-0.2,-2/3-0.2) - - (2/3+4.2-0.15,-2/3-0.2)- - (2/3+4.2-0.15,-2/3+0.05)- - (2/3+4.2+0.15,-2/3+0.05)- - (2/3+4.2+0.15,-2/3-0.3)- - (-2/3+4.2-0.3,-2/3-0.3)- - (-2/3+4.2-0.3,2/3-0.25)- - (-2/3+4.2+0.35,2/3-0.25)- - (-2/3+4.2+0.35,2/3+0.45)- - (0+4.2,2/3+0.45)- - (0+4.2,2);

\filldraw [blue] (-2/3+4.2-0.1,2/3-0.05) circle (0.8pt);
\filldraw [blue] (-2/3+4.2-0.1,2/3-0.1) circle (0.8pt);
\filldraw [blue] (-2/3+4.2-0.1,2/3-0.2) circle (0.8pt);
\filldraw [blue] (-2/3+4.2-0.1,2/3-0.25) circle (0.8pt);
\filldraw [blue] (-2/3+4.2-0.1,-2/3+0.25) circle (0.8pt);
\filldraw [blue] (-2/3+4.2-0.1,-2/3+0.17) circle (0.8pt);

\filldraw [blue] (2/3+4.2+0.1,2/3-0.05) circle (0.8pt);
\filldraw [blue] (2/3+4.2+0.1,-2/3+0.1) circle (0.8pt);
\filldraw [blue] (2/3+4.2+0.1,-2/3+0.03) circle (0.8pt);
;

\filldraw [blue] (-2/3+4.2+0.13,2/3+0.2) circle (0.8pt);
\filldraw [blue] (-2/3+4.2+0.2,2/3+0.2) circle (0.8pt);
\filldraw [blue] (-2/3+4.2+0.3,2/3+0.2) circle (0.8pt);
\filldraw [blue] (-2/3+4.2+0.35,2/3+0.2) circle (0.8pt);

\filldraw [blue] (2/3+4.2-0.13,2/3+0.2) circle (0.8pt);
\filldraw [blue] (2/3+4.2-0.25,2/3+0.2) circle (0.8pt);
\filldraw [blue] (2/3+4.2-0.3,2/3+0.2) circle (0.8pt);

  \end{tikzpicture}
  }
  \caption{The polynomial corresponding to the intersection points marked on the left side is $P(t,t^{-1})=t^{-5}-t^{-4}+t^{-2}-2t^{-1}+1$. The polynomial corresponding to the marked intersection points above puncture points is $Q(t,t^{-1})=-t^{-4}+t^{-3}-2t^{-1}+2-t$ and the polynomial corresponding to the marked intersection points on the right side is $R(t,t^{-1})=1-t^2+t^3$}
  \label{figure 16}
\end{figure}
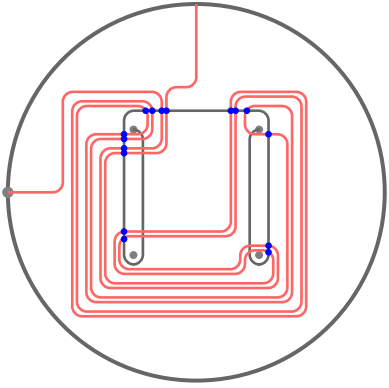

\[{\rho }_{11}\left(a^4ab^2ab^{-1}\right)=\left(t^{-5}-t^{-4}+t^{-2}-2t^{-1}+1\right)\left(1-t^{-1}+t^{-2}-t^{-3}\right)\] 
\[\left(-t^{-4}+t^{-3}-2t^{-1}+2-t\right)+\left(1-t^2+t^3\right)\left(1-t+t^2-t^3\right)=\] 
\[-t^{-8}+2t^{-7}-2t^{-6}+t^{-5}+t^{-4}-3t^{-3}+4t^{-2}-5t^{-1}+4\] 
\[-2t+t^3-2t^4+2t^5-t^6,\] 
\[{\rho }_{13}\left(a^4ab^2ab^{-1}\right)=-\left(t^{-5}-t^{-4}+t^{-2}-2t^{-1}+1\right)\left(1-t^{-1}+t^{-2}\right)-\] 
\[\left(-t^{-4}+t^{-3}-2t^{-1}+2-t\right)-\left(1-t^2+t^3\right)\left(1-t+t^2\right)=\] 
\[-t^{-7}+2t^{-6}-2t^{-5}+t^{-4}+2t^{-3}-4t^{-2}+5t^{-1}-4+2t-2t^3+2t^4-t^5.\]

We formulate a conjecture that describes a certain regularity, experimentally observed for images (matrices) of braids of a special form. For future reference let us state clearly that what we mean by {\sl regularity} is that the $(1,1)$ and 
$(3,1)$ entries in the considered matrix are non--zero Laurent polynomials and that the difference of the degrees of lowest degree terms is equal to $-1$.
 
\begin{conjecture} Let $\ \sigma \in B_4$ be any non--trivial pure braid which is not equivalent to $ \Delta^m,$ for some $m \in \mathbb{N}.$ We assume that $\sigma $ acts non--trivially on $T_4.$ Then there exists a sufficiently large $l_0 \in \mathbb{N}$ with respect to the length of $ \sigma $ and a sufficiently large $m_0\in \mathbb{N}$ such that for each $m > m_0, l \ge l_0$ the difference of the lowest degrees of the polynomials 
${\rho }_{11}(a^m\sigma a^{-l})$ and ${\rho }_{31}(a^m\sigma a^{-l})$ 
is equal to $-1$ and the polynomials are non-zero.
\end{conjecture}

While experimental data suggest that the Conjecture is true as formulated, we are really interested in the situation when $\sigma$ is a product of Bokut--Vesnin generators. Therefore we may refer to the length of $\sigma$, meaning the length of
$\sigma$ as a reduced word in $a,b,a^{-1},b^{-1}$.
Now, we give some arguments showing why we expect the Conjecture to be true. Take a sufficiently large $l_0 \in \mathbb{N}$ with respect to the length of $ \sigma $. Consider the curves $N_1a^{l_0}$ and neighborhood  $U_1$ of  $T_4$ inside of which it looks~as in Figure~17:

\begin{figure}[ht]
  \centering
  \resizebox{150pt}{!}{%
  \begin{tikzpicture}
  
\filldraw  [color=black!60, fill=white, very thick] (0,0) circle (2);

\filldraw [gray] (-2,0) circle (1.5pt); 
\filldraw [gray] (-2/3,2/3) circle (1pt);    
\filldraw [gray] (2/3,2/3) circle (1pt);    
\filldraw [gray] (2/3,-2/3) circle (1pt);    
\filldraw [gray] (-2/3,-2/3) circle (1pt);

\draw[color=black!60, dashed](-2/3,0) ellipse (0.85 and 1.45);

\draw  [color=red!60, thick, rounded corners=3pt] (-1.55,0) - - (-2/3-0.5,0) - - (-2/3-0.5,-2/3-0.4)- - (-2/3+0.425,-2/3-0.4)- - (-2/3+0.425,2/3+0.35)- - (-2/3-0.35,2/3+0.35)- - (-2/3-0.35,-2/3-0.25)- - (-2/3+0.275,-2/3-0.25)- - (-2/3+0.275,2/3+0.2)- - (-2/3-0.2,2/3+0.2)- - (-2/3-0.2,0.2) ;
\draw  [color=red!60, thick, rounded corners=2pt] (-2/3-0.2,-0.2) - - (-2/3-0.2,-2/3-0.125)- - (-2/3+0.15,-2/3-0.125)- - (-2/3+0.15,2/3+0.075)- - (-2/3-0.1,2/3+0.075)- - (-2/3-0.1,2/3-0.075)- - (-2/3+0.1,2/3-0.075)- - (-2/3+0.1,-2/3-0.075)- - (-2/3-0.15,-2/3-0.075)- - (-2/3-0.15,-0.2)

;
\draw  [color=red!60, thick, rounded corners=3pt]  (-2/3-0.15,0.2) - - (-2/3-0.15,2/3+0.15)- - (-2/3+0.225,2/3+0.15)- - (-2/3+0.225,-2/3-0.2)- - (-2/3-0.3,-2/3-0.2)- - (-2/3-0.3,2/3+0.3)- - (-2/3+0.375,2/3+0.3)- - (-2/3+0.375,-2/3-0.35)- - (-2/3-0.45,-2/3-0.35)- - (-2/3-0.45,0.05)- - (-1.55, 0.05);

\filldraw [red] (-2/3-0.175,0.1) circle (0.5pt);
\filldraw [red] (-2/3-0.175,0) circle (0.5pt);
\filldraw [red] (-2/3-0.175,-0.1) circle (0.5pt);
  \end{tikzpicture}
  }
\caption{}
\label{figure 17}
\end{figure}

Apply to the curve $N_1a^{l_0}$ the transformation corresponding to the braid $ \sigma ^{-1} $. Note that $ l_0$ is sufficiently large with respect to the length of $ \sigma ^{-1}$ and so in the curve $N_1a^{l_0} \sigma ^{-1}$ almost all parallel lines to line $T_4$ are followed by the curve $T_4 \sigma ^{-1}$. On the other hand the transformation corresponding to the braid $ \sigma ^{-1} $ acts non--trivially on the line $T_4$. Therefore the final image $N_1a^{l_0} \sigma ^{-1}$ does not have problematic strings around $T_4$, as in Figure~18:

\begin{figure}[ht]
  \centering
  \resizebox{320pt}{!}{%
  \begin{tikzpicture}
  
\filldraw  [color=black!60, fill=white, very thick] (3,0) circle (2);

\filldraw [gray] (-2+3,0) circle (1pt); 
\filldraw [gray] (-2/3+3,2/3) circle (1pt);    
\filldraw [gray] (2/3+3,2/3) circle (1pt);    
\filldraw [gray] (2/3+3,-2/3) circle (1pt);    
\filldraw [gray] (-2/3+3,-2/3) circle (1pt);

\draw[color=black!60, dashed](-2/3+3,0) ellipse (0.5 and 1);
\draw[color=black!60, dashed](2/3+3,0) ellipse (0.5 and 1);

\draw  [color=red!60, thick, rounded corners=2pt]  (-2/3+3-0.15,-0.15) - - (-2/3+3-0.15,-2/3-0.125)- - (-2/3+3+0.15,-2/3-0.125)- - (-2/3+3+0.15,2/3+0.075)- - (-2/3+3-0.1,2/3+0.075)- - (-2/3+3-0.1,2/3-0.075)- - (-2/3+3+0.1,2/3-0.075)- - (-2/3+3+0.1,-2/3-0.075)- - (-2/3+3-0.1,-2/3-0.075)- - (-2/3+3-0.1,-0.15);

\filldraw [red] (-2/3+3-0.125,0.1) circle (0.5pt);
\filldraw [red] (-2/3+3-0.125,0) circle (0.5pt);
\filldraw [red] (-2/3+3-0.125,-0.1) circle (0.5pt);

\filldraw  [color=black!60, fill=white, very thick] (-3,0) circle (2);

\filldraw [gray] (-2-3,0) circle (1pt); 
\filldraw [gray] (-2/3-3,2/3) circle (1pt);    
\filldraw [gray] (2/3-3,2/3) circle (1pt);    
\filldraw [gray] (2/3-3,-2/3) circle (1pt);    
\filldraw [gray] (-2/3-3,-2/3) circle (1pt);

\draw[color=black!60, dashed](-2/3-3,0) ellipse (0.5 and 1);
\draw[color=black!60, dashed](2/3-3,0) ellipse (0.5 and 1);

\draw  [color=red!60, thick, rounded corners=2pt]   (-2/3-3+0.15,0.15) - - (-2/3-3+0.15,2/3+0.15)- - (-2/3-3-0.125,2/3+0.15)- - (-2/3-3-0.125,-2/3-0.075)- - (-2/3-3+0.1,-2/3-0.075)- - (-2/3-3+0.1,-2/3+0.075)- - (-2/3-3-0.075,-2/3+0.075)- - (-2/3-3-0.075,2/3+0.1) - -(-2/3-3+0.1,2/3+0.1)- -(-2/3-3+0.1,0.15);

\filldraw [red] (-2/3-3+0.125,0.1) circle (0.5pt);
\filldraw [red] (-2/3-3+0.125,0) circle (0.5pt);
\filldraw [red] (-2/3-3+0.125,-0.1) circle (0.5pt);
  \end{tikzpicture}
  }
  \caption{}
  \label{figure 16}
\end{figure}

In general, if we take any braid $\sigma$, then $F_1\sigma$ may have the strings in the form illustrated in Figure~18. For example if $\sigma=b^{-1}ab^{-1}$ or $\sigma=a^3$ then the corresponding curves are shown in Figure~19.

Because the curve $N_1a^{l_0} \sigma ^{-1}$  does not have any problematic strings for $n=3$ the intersection of curves $F_1a^3$ and $N_1a^{l_0} \sigma ^{-1}$  inside the neighbourhoods $U_1$ and $U_2$ of $T_4$ and $T_2$ looks as in Figure~13. So the ${\rho }_{11}(a^{3}\sigma a^{-{l_o}})$ and ${\rho }_{31}(a^{3}\sigma a^{-{l_0}})$  entries of the Burau matrix $\rho \left(a^n\sigma a^{-{l_0}}\right)$ can be written as 
\[{\rho }_{11}\left(a^3\sigma a^{-{l_0}}\right)=P\left(t,t^{-1}\right)\left(1-t^{-1}\right)+Q\left(t,t^{-1}\right)+R\left(t,t^{-1}\right)(1-t),\] 
\[{\ \rho }_{31}\left(a^3\sigma a^{-{l_0}}\right)=-P\left(t,t^{-1}\right)-Q\left(t,t^{-1}\right)-R\left(t,t^{-1}\right),\] 
where polynomial $P\left(t,t^{-1}\right)$ is not zero and for each $m' \in \mathbb{N}$ we have\\
$\begin{array}{lclll}
{\rho }_{11}\left(a^{3+m'}\sigma a^{-{l_0}}\right)&=&\hphantom{-}P\left(t,t^{-1}\right)\left(S_{m'+1}(t^{-1})\right)&+
Q\left(t,t^{-1}\right)\cr
&&+R\left(t,t^{-1}\right)\left(S_{m'+1}(t)\right),\cr
{\rho }_{31}\left(a^{3+m'}\sigma a^{-l_0}\right)&=&-P\left(t,t^{-1}\right)\left(S_{m'}(t^{-1})\right)&-
Q\left(t,t^{-1}\right)\cr
&&-R\left(t,t^{-1}\right)\left(S_{m'}(t)\right).\cr
\end{array}
$
\begin{figure}[ht]
  \centering
  \resizebox{320pt}{!}{%
  \begin{tikzpicture}
  
\filldraw  [color=black!60, fill=white, very thick] (-3,0) circle (2);

\filldraw [gray] (-2-3,0) circle (1pt); 
\filldraw [gray] (-2/3-3,2/3) circle (1pt);    
\filldraw [gray] (2/3-3,2/3) circle (1pt);    
\filldraw [gray] (2/3-3,-2/3) circle (1pt);    
\filldraw [gray] (-2/3-3,-2/3) circle (1pt);

\draw  [color=red!60, thick,  rounded corners=2pt]  (-2-3,0) - -  (-2/3-3-0.4,0) - -  (-2/3-3-0.4,2/3+0.4)- -  (2/3-3+0.1,2/3+0.4)- -  (2/3-3+0.1,2/3-0.1)- -  (-2/3-3+0.2,2/3-0.1)- -  (-2/3-3+0.2,2/3+0.2)- -  (-2/3-3-0.2,2/3+0.2)- -  (-2/3-3-0.2,-2/3-0.1)- -  (-2/3-3+0.1,-2/3-0.1)- -  (-2/3-3+0.1,-2/3+0.1)- -  (-2/3-3-0.1,-2/3+0.1)- -  (-2/3-3-0.1,2/3+0.1)- -  (-2/3-3+0.1,2/3+0.1)- -  (-2/3-3+0.1,2/3-0.2)- -  (2/3-3+0.2,2/3-0.2)- -  (2/3-3+0.2,2/3+0.5)- -  (0-3,2/3+0.5)- -  (0-3,2);

\filldraw  [color=black!60, fill=white, very thick] (3,0) circle (2);

\filldraw [gray] (-2+3,0) circle (1pt); 
\filldraw [gray] (-2/3+3,2/3) circle (1pt);    
\filldraw [gray] (2/3+3,2/3) circle (1pt);    
\filldraw [gray] (2/3+3,-2/3) circle (1pt);    
\filldraw [gray] (-2/3+3,-2/3) circle (1pt);

\draw  [color=red!60, thick,  rounded corners=2pt]  (-2+3,0) - -  (-2/3+3-0.5,0)- -  (-2/3+3-0.5,-2/3-0.4)- -  (-2/3+3+0.4,-2/3-0.4)- -  (-2/3+3+0.4,2/3+0.2)- -  (-2/3+3-0.2,2/3+0.2)- -  (-2/3+3-0.2,-2/3-0.1)- -  (-2/3+3+0.1,-2/3-0.1)- -  (-2/3+3+0.1,-2/3+0.1)- -  (-2/3+3-0.1,-2/3+0.1)- -  (-2/3+3-0.1,2/3+0.1)- -  (-2/3+3+0.3,2/3+0.1)- -  (-2/3+3+0.3,-2/3-0.3)- -  (-2/3+3-0.4,-2/3-0.3)- -  (-2/3+3-0.4,2/3+0.4)- -  (0+3,2/3+0.4)- -  (0+3,2);

  \end{tikzpicture}
  }
\caption{}
\label{figure 17}
\end{figure}

On the other hand if we compare the curves $N_1 a^{{l_0}+1}$ and $N_1 a^{l_0}$ (see Figure~20), it is clear that they differ only by strings around $T_4$.

\begin{figure}[ht]
  \centering
  \resizebox{330pt}{!}{%
  \begin{tikzpicture}
  
\filldraw  [color=black!60, fill=white, very thick] (-3,0) circle (2);

\filldraw [gray] (-2-3,0) circle (1pt); 
\filldraw [gray] (-2/3-3,2/3) circle (1pt);    
\filldraw [gray] (2/3-3,2/3) circle (1pt);    
\filldraw [gray] (2/3-3,-2/3) circle (1pt);    
\filldraw [gray] (-2/3-3,-2/3) circle (1pt);

\draw[color=black!60, dashed](-2/3-3,0) ellipse (0.8 and 1.45);

\draw  [color=red!60, thick, rounded corners=2pt] (-1.4-3,0) - - (-2/3-3-0.5,0) - - (-2/3-3-0.5,-2/3-0.4)- - (-2/3-3+0.4,-2/3-0.4)- - (-2/3-3+0.4,-2/3-0.4)- - (-2/3-3+0.4,-2/3-0.4)- - (-2/3-3+0.4,2/3+0.35)- - (-2/3-3-0.35,2/3+0.35)- - (-2/3-3-0.35,-2/3-0.25)- - (-2/3-3+0.25,-2/3-0.25)- - (-2/3-3+0.25,2/3+0.2)- - (-2/3-3-0.2,2/3+0.2)- - (-2/3-3-0.2,0.2) ;
\draw  [color=red!60, thick, rounded corners=2pt] (-2/3-3-0.2,-0.2) - - (-2/3-3-0.2,-2/3-0.115)- - (-2/3-3+0.125,-2/3-0.125)- - (-2/3-3+0.125,2/3+0.075)- - (-2/3-3-0.075,2/3+0.075)- - (-2/3-3-0.075,2/3-0.075)- - (-2/3-3+0.075,2/3-0.075)- - (-2/3-3+0.075,-2/3-0.075)- - (-2/3-3-0.15,-2/3-0.075)- - (-2/3-3-0.15,-0.2);
\draw  [color=red!60, thick, rounded corners=2pt]  (-2/3-3-0.15,0.2) - - (-2/3-3-0.15,2/3+0.15)- - (-2/3-3+0.2,2/3+0.15)- - (-2/3-3+0.2,-2/3-0.2)- - (-2/3-3-0.3,-2/3-0.2)- - (-2/3-3-0.3,2/3+0.3)- - (-2/3-3+0.35,2/3+0.3)- - (-2/3-3+0.35,-2/3-0.35)- - (-2/3-3-0.45,-2/3-0.35)- - (-2/3-3-0.45,0.05)- - (-1.4-3, 0.05);

\filldraw [red] (-2/3-3-0.175,0.1) circle (0.5pt);
\filldraw [red] (-2/3-3-0.175,0) circle (0.5pt);
\filldraw [red] (-2/3-3-0.175,-0.1) circle (0.5pt);

\filldraw  [color=black!60, fill=white, very thick] (3,0) circle (2);

\filldraw [gray] (-2+3,0) circle (1pt); 
\filldraw [gray] (-2/3+3,2/3) circle (1pt);    
\filldraw [gray] (2/3+3,2/3) circle (1pt);    
\filldraw [gray] (2/3+3,-2/3) circle (1pt);    
\filldraw [gray] (-2/3+3,-2/3) circle (1pt);

\draw[color=black!60, dashed](-2/3+3,0) ellipse (0.8 and 1.45);

\draw  [color=red!60, thick, rounded corners=2pt] (-1.4+3,0) - - (-2/3+3-0.5,0) - - (-2/3+3-0.5,-2/3-0.4)- - (-2/3+3+0.4,-2/3-0.4)- - (-2/3+3+0.4,2/3+0.35)- - (-2/3+3-0.35,2/3+0.35)- - (-2/3+3-0.35,-2/3-0.25)- - (-2/3+3+0.25,-2/3-0.25)- - (-2/3+3+0.25,2/3+0.2)- - (-2/3+3-0.2,2/3+0.2)- - (-2/3+3-0.2,0.2) ;
\draw  [color=red!60, thick, rounded corners=2pt] (-2/3+3-0.2,-0.2) - - (-2/3+3-0.2,-2/3-0.075)- - (-2/3+3+0.075,-2/3-0.075)- - (-2/3+3+0.075,-2/3+0.075)- - (-2/3+3-0.15,-2/3+0.075)- - (-2/3+3-0.15,-0.2);
\draw  [color=red!60, thick, rounded corners=2pt]  (-2/3+3-0.15,0.2) - - (-2/3+3-0.15,2/3+0.15)- - (-2/3+3+0.2,2/3+0.15)- - (-2/3+3+0.2,-2/3-0.2)- - (-2/3+3-0.3,-2/3-0.2)- - (-2/3+3-0.3,2/3+0.3)- - (-2/3+3+0.35,2/3+0.3)- - (-2/3+3+0.35,-2/3-0.35)- - (-2/3+3-0.45,-2/3-0.35)- - (-2/3+3-0.45,0.05)- - (-1.4+3, 0.05);

\filldraw [red] (-2/3+3-0.175,0.1) circle (0.5pt);
\filldraw [red] (-2/3+3-0.175,0) circle (0.5pt);
\filldraw [red] (-2/3+3-0.175,-0.1) circle (0.5pt);
  \end{tikzpicture}
  }
  \caption{}
  \label{figure 18}
\end{figure}

Moreover, if we consider the intersections of curves $N_1a^{l_0 +1} $ and $N_1a^{l_0}$ with the strings between the puncture points $p_1$ and $p_4$ as in Figure~21, then corresponding polynomials  up to sign $\epsilon$ and multiplications $t^\alpha$ have the forms:

 \ \

\[\left(\epsilon t^\alpha+S(t,t^{-1})(1-t^{-1})\right)\left(1-t+\dots +{\left(-1\right)}^lt^{l-1}\right),                                (*)\] 
\[\left(\epsilon t^\alpha+S(t,t^{-1})(1-t^{-1})\right)\left(1-t+\dots +{\left(-1\right)}^{l+1}t^{l-2}\right).                           (**)\] 

Therefore their lowest degrees are equal.  Note that, $l_0$ is sufficiently large with respect to the length of $ \sigma $ and so the pictures of curves $N_1a^{l_0 +1} \sigma ^{-1} $ and  $ N_1a^{l_0} \sigma ^{-1} $ `globally' are the same. That means that in some 'local' pictures there are just different numbers of strings. Now we must look at pictures $N_1a^{l_0 +1} \sigma ^{-1}$ and  $N_1a^{l_0} \sigma ^{-1} $ inside a neighborhood of  $T$ as it was done in the previous proof (see Figure~10). Note that by the arguments in the proof of Lemma 4.1 and same `global' picture of the curves $ N_1a^{l_0 +1} \sigma ^{-1} $ and  $N_1 a^{l_0} \sigma ^{-1} $  we have

\[{\rho }_{11}\left(a^3\sigma a^{-l_0 -1}\right)=P'\left(t,t^{-1}\right)\left(1-t^{-1}\right)+Q'\left(t,t^{-1}\right)+R'\left(t,t^{-1}\right)(1-t),\] 
\[{\ \rho }_{31}\left(a^3\sigma a^{-l_0 -1}\right)=-P'\left(t,t^{-1}\right)-Q'\left(t,t^{-1}\right)-R'\left(t,t^{-1}\right),\] 
and for each $m' \in \mathbb{N}$ we have:\\
$\begin{array}{lclll}
{\rho }_{11}\left(a^{3+m'}\sigma a^{-l_0 -1}\right)&=&\hphantom{-}P'\left(t,t^{-1}\right)\left(S_{m'+1}(t^{-1})\right)&+Q'\left(t,t^{-1}\right)\cr
&&+R'\left(t,t^{-1}\right)\left(S_{m'+}(t)\right),\cr
{\rho }_{31}\left(a^{3+m'}\sigma a^{-l_0 -1}\right)&=&-P'\left(t,t^{-1}\right)\left(S_{m'}(t^{-1})\right)&-Q'\left(t,t^{-1}\right)
\cr
&&-R'\left(t,t^{-1}\right)\left(S_{m'+}(t)\right).
\end{array}
$

\begin{figure}[ht]
  \centering
  \resizebox{330pt}{!}{%
  \begin{tikzpicture}
  
\filldraw  [color=black!60, fill=white, very thick] (-3,0) circle (2);

\filldraw [gray] (-2-3,0) circle (1pt); 
\filldraw [gray] (-2/3-3,2/3) circle (1pt);    
\filldraw [gray] (2/3-3,2/3) circle (1pt);    
\filldraw [gray] (2/3-3,-2/3) circle (1pt);    
\filldraw [gray] (-2/3-3,-2/3) circle (1pt);

\draw[color=black!60, dashed](-2/3-3,0) ellipse (0.8 and 1.45);

\draw  [color=red!60, thick, rounded corners=2pt] (-1.4-3,0) - - (-2/3-3-0.5,0) - - (-2/3-3-0.5,-2/3-0.4)- - (-2/3-3+0.4,-2/3-0.4)- - (-2/3-3+0.4,-2/3-0.4)- - (-2/3-3+0.4,-2/3-0.4)- - (-2/3-3+0.4,2/3+0.35)- - (-2/3-3-0.35,2/3+0.35)- - (-2/3-3-0.35,-2/3-0.25)- - (-2/3-3+0.25,-2/3-0.25)- - (-2/3-3+0.25,2/3+0.2)- - (-2/3-3-0.2,2/3+0.2)- - (-2/3-3-0.2,0.2) ;
\draw  [color=red!60, thick, rounded corners=2pt] (-2/3-3-0.2,-0.2) - - (-2/3-3-0.2,-2/3-0.125)- - (-2/3-3+0.125,-2/3-0.125)- - (-2/3-3+0.125,2/3+0.075)- - (-2/3-3-0.075,2/3+0.075)- - (-2/3-3-0.075,2/3-0.075)- - (-2/3-3+0.075,2/3-0.075)- - (-2/3-3+0.075,-2/3-0.075)- - (-2/3-3-0.15,-2/3-0.075)- - (-2/3-3-0.15,-0.2);
\draw  [color=red!60, thick, rounded corners=2pt]  (-2/3-3-0.15,0.2) - - (-2/3-3-0.15,2/3+0.15)- - (-2/3-3+0.2,2/3+0.15)- - (-2/3-3+0.2,-2/3-0.2)- - (-2/3-3-0.3,-2/3-0.2)- - (-2/3-3-0.3,2/3+0.3)- - (-2/3-3+0.35,2/3+0.3)- - (-2/3-3+0.35,-2/3-0.35)- - (-2/3-3-0.45,-2/3-0.35)- - (-2/3-3-0.45,0.05)- - (-1.4-3, 0.05);

\filldraw [red] (-2/3-3-0.175,0.1) circle (0.5pt);
\filldraw [red] (-2/3-3-0.175,0) circle (0.5pt);
\filldraw [red] (-2/3-3-0.175,-0.1) circle (0.5pt);

\draw  [color=black!60,thick] (-2/3-3,2/3) - - (-2/3-3+0.6,2/3);
\draw  [color=black!60, thick] (-2/3-3-0.7,-0.1) - - (-2/3-3+0.7,-0.1);

\filldraw  [color=black!60, fill=white, thick] (3,0) circle (2);

\filldraw [gray] (-2+3,0) circle (1pt); 
\filldraw [gray] (-2/3+3,2/3) circle (1pt);    
\filldraw [gray] (2/3+3,2/3) circle (1pt);    
\filldraw [gray] (2/3+3,-2/3) circle (1pt);    
\filldraw [gray] (-2/3+3,-2/3) circle (1pt);

\draw[color=black!60, dashed](-2/3+3,0) ellipse (0.8 and 1.45);

\draw  [color=red!60, thick, rounded corners=2pt] (-1.4+3,0) - - (-2/3+3-0.5,0) - - (-2/3+3-0.5,-2/3-0.4)- - (-2/3+3+0.4,-2/3-0.4)- - (-2/3+3+0.4,2/3+0.35)- - (-2/3+3-0.35,2/3+0.35)- - (-2/3+3-0.35,-2/3-0.25)- - (-2/3+3+0.25,-2/3-0.25)- - (-2/3+3+0.25,2/3+0.2)- - (-2/3+3-0.2,2/3+0.2)- - (-2/3+3-0.2,0.2) ;
\draw  [color=red!60, thick, rounded corners=2pt] (-2/3+3-0.2,-0.2) - - (-2/3+3-0.2,-2/3-0.075)- - (-2/3+3+0.075,-2/3-0.075)- - (-2/3+3+0.075,-2/3+0.075)- - (-2/3+3-0.15,-2/3+0.075)- - (-2/3+3-0.15,-0.2);
\draw  [color=red!60, thick, rounded corners=2pt]  (-2/3+3-0.15,0.2) - - (-2/3+3-0.15,2/3+0.15)- - (-2/3+3+0.2,2/3+0.15)- - (-2/3+3+0.2,-2/3-0.2)- - (-2/3+3-0.3,-2/3-0.2)- - (-2/3+3-0.3,2/3+0.3)- - (-2/3+3+0.35,2/3+0.3)- - (-2/3+3+0.35,-2/3-0.35)- - (-2/3+3-0.45,-2/3-0.35)- - (-2/3+3-0.45,0.05)- - (-1.4+3, 0.05);

\filldraw [red] (-2/3+3-0.175,0.1) circle (0.5pt);
\filldraw [red] (-2/3+3-0.175,0) circle (0.5pt);
\filldraw [red] (-2/3+3-0.175,-0.1) circle (0.5pt);

\draw  [color=black!60,thick] (-2/3+3,2/3) - - (-2/3+3+0.6,2/3);
\draw  [color=black!60, thick] (-2/3+3-0.7,-0.1) - - (-2/3+3+0.7,-0.1);

  \end{tikzpicture}
  }
  \caption{}
  \label{figure 19}
\end{figure}

Take a large $m_0=3+m'$ such that the difference of lowest degrees of polynomials  ${\rho }_{11} \left( a^{m_0} \sigma a^{-{l_0}} \right)$ and ${\rho }_{31} \left( a^{m_0} \sigma a^{-{l_0}} \right)$ is equal to $-1$ and these lowest degrees come from the lowest degree of the polynomial $P(t,t^{-1})$. By (*) and (**) the polynomials  $P\left(t,t^{-1}\right)$ and $P'\left(t,t^{-1}\right)$, $Q\left(t,t^{-1}\right)$ and $Q'\left(t,t^{-1}\right)$ and also $R\left(t,t^{-1}\right)$ and $R'\left(t,t^{-1}\right)$ have the same lowest degrees and so the same regularity will be true for the polynomials ${\rho }_{11}\left(a^{m}\sigma ba^{-l_0 -1}\right)$ and ${\rho }_{31}\left(a^{m}\sigma a^{-l_0 -1}\right)$. By 
induction on the length of $\sigma$ it will be true for the braid $a^m\sigma a^{-l}, l>l_0$ 
as well.

\textbf{Example 4.} Let $\sigma =\ b^6a\ b^{-1}a^{-1}b^{-6}a^{-6}$. Our aim is to find $n$ which satisfies the conditions of 
Lemma~4.1 and to calculate the corresponding  polynomials $P$, $Q$ and $R$. 
Then we will take any $l>6$ (In our case we consider $l=9$) and will show that lowest degrees of the corresponding polynomials do not change. For the given braid it is difficult to see the picture and write down the polynomials $P$, $Q$ and $R$. Therefore we will use the following method:
If $n$ (in our case $n=2$) is a number as in Lemma~4.1, then we have 

\[{\rho}_{11}\left(\sigma\right)+{\rho }_{31}\left(\sigma\right)=-t^{-1}P(t,t^{-1})-tR(t,t^{-1}),\]
\[{\rho }_{11}\left(a\sigma\right)+{\rho }_{31}\left(a\sigma\right)=t^{-2}P(t,t^{-1})+t^2R(t,t^{-1}).\]
Therefore
\[P(t,t^{-1})=\frac{t({\rho }_{11}(\sigma)+{\rho }_{31}(\sigma))+({\rho }_{11}(a\sigma)+{\rho }_{31}(a\sigma))}{t^{-2}-1} ,\]
\[R(t,t^{-1})= \frac{t^{-1}({\rho }_{11}(\sigma)+{\rho }_{31}(\sigma))+({\rho }_{11}(a\sigma)+{\rho }_{31}(a\sigma))}{t^{2}-1}.\]
In this way we can see that for the braid $\sigma =a^2b^6ab^{-1}a^{-1}b^{-6}a^{-6}$ we have
\[P\left(t,t^{-1}\right)=t^{-8}-3t^{-7}+6t^{-6}-9t^{-5}+11t^{-4}-11t^{-3}+8t^{-2}-2t^{-1}-\] 
\[-7+16t^1-22t^2+23t^3-20t^4+14t^5-5t^6-4t^7+\]
\[+10t^8-12t^9+11t^{10}-9t^{11}+6t^{12}-3t^{13}+t^{14}\]
 
\[Q\left(t,t^{-1}\right)=-t^{-1}+2-2t^1+t^2-2t^4+4t^5-6t^6+6t^7-4t^8+t^9+\]
\[+t^{10}-2t^{11}+3t^{12}-3t^{13}+2t^{14}-t^{15},\] 

\[R\left(t,t^{-1}\right)=-t^{-6}+3t^{-5}-6t^{-4}+9t^{-3}-11t^{-2}+11t^{-1}-\] 
\[-7-t^1+10t^2-18t^3+23t^4-22t^5+17t^6-8t^7-t^8+8t^9-11t^{10}+\]
\[+11t^{11}-9t^{12}+6t^{13}-3t^{14}+t^{15}.\] 

Similarly, for the braid $\sigma' =a^2b^6ab^{-1}a^{-1}b^{-6}a^{-9}$ we obtain

\[P'\left(t,t^{-1}\right)=t^{-8}-3t^{-7}+6t^{-6}-9t^{-5}+11t^{-4}-12t^{-3}+11t^{-2}-8t^{-1}+\] 
\[+1+8t^1-16t^2+21t^3-23t^4+23t^5-19t^6+12t^7-\]
\[-4t^8-3t^9+8t^{10}-11t^{11}+12t^{12}-11t^{13}+9t^{14}-6t^{15}+3t^{16}-t^{17}\]

\[Q'\left(t,t^{-1}\right)=-t^{-1}+2-2t^1+t^2-t^4+2t^5-4t^6+6t^7-6t^8+4t^9-\]
\[-2t^{10}+t^{11}-t^{13}+2t^{14}-3t^{15}+3t^{16}-2t^{17}+t^{18},\] 

\[R'\left(t,t^{-1}\right)=-t^{-6}+3t^{-5}-6t^{-4}+9t^{-3}-11t^{-2}+12t^{-1}-\] 
\[-10+5t^1+2t^2-10t^3+17t^4-21t^5+22t^6-20t^7+14t^8-7t^9+\]
\[+5t^{11}-9t^{12}+11t^{13}-11t^{14}+9t^{15}-6t^{16}+3t^{17}-t^{18}.\] 
Therefore the lowest degrees of polynomials $P\left(t,t^{-1}\right)$ and 
$P'\left(t,t^{-1}\right)$ (same situation is with the polynomials 
$Q\left(t,t^{-1}\right)$ and $Q'\left(t,t^{-1}\right)$ or $R\left(t,t^{-1}\right)$ 
and $R'\left(t,t^{-1}\right)$) are equal.
\noindent 

\begin{theorem} Conjecture~4.2 implies faithfulness of the Burau representation for  $n=4$ .\end{theorem}

\begin{proof} Let us consider a nontrivial braid $\sigma \in B_4$ written as a reduced word in the Bokut--Vesnin generators. We may assume that it begins and ends with $a$ or $a^{-1}$ 
(otherwise we will conjugate by a suitable power of $a$). 

If we interpret the braid group as the mapping class group, then there is a natural induced action on the set of isotopy classes of forks. Let $\sigma $ act non--trivially on $T_4$. Then by Lemma~4.1 it is possible to find sufficiently large $l_0$ with respect to the length of $\sigma $ and sufficiently large $m_0$, such that for each $m > m_0$ and $l >l_0$ the difference of 
lowest degrees of the polynomials ${\rho }_{11}(a^m\sigma a^{-l})$
and ${\rho }_{31}(a^m \sigma a^{-l})$ is equal to $-1$ and the polynomials
are both non-zero. In particular, we can assume that $m=l$ and so 
${\rho }_{11}(a^m \sigma a^{-m})$ and ${\rho }_{31}(a^m \sigma a^{-m})$ are both non-zero
which contradicts the assumption that 
$a^m \sigma a^{-m}\in \ker\rho$.

The general case (when we do not assume that $\sigma$ acts non--trivially on $T_4$) is easily reduced to the one discussed above. The reason is that if $\sigma$ acts trivially on all four segments, then $\sigma$ is a power of $\Delta$ which is not possible if $\sigma$ is a product of the Bokut--Vesnin generators. And if $\sigma$ acts non--trivially on at least one of the four segments, then we can rotate the whole disc to make the action non--trivial for $T_4$.

\end{proof}

\begin{remark} We have a C\texttt{++} program checking whether  our regularity works or not for randomly generated examples. We calculated millions of examples and the regularity was always confirmed. In fact we considered examples of type 
$a^3b^3 w b^{-3}a^{-3}$, where $a^3b^3 w b^{-3}a^{-3}$ is a reduced word in the Bokut--Vesnin generators. Such a version of Proposition 3.1 is sufficient for the Burau representation faithfulness problem.
\end{remark}

\section*{Acknowledgements:} Most of this research was conducted while the first author was a postdoc at the University of Warsaw during the Spring 2014 semester, with support of Erasmus Mundus Project (WEBB).

\noindent 
\begin{center}
 REFERENCES
\end{center}

\begin{enumerate}
\item  \textbf{Joan S Birman.} Braids, links, and mapping class groups. \textit{Annals of Mathematics Studies, No. 82, Princeton University Press, Princeton, NJ (1974)}\textbf{\textit{}}

\item \textbf{Stephen Bigelow.} The Burau representation is not faithful for $n=5$. \textit{Geom. Topol.} 3 (1999), 397-404

\item  \textbf{Stephen Bigelow.} Does the Jones polynomial detect the unknot? \textit{J.~Knot Theory and Ramifications \textit{(4)} 11 (2002), 493-505}

\item \textbf{Leonid Bokut and Andrei Vesnin.} New rewriting system for the braid group $B_4$.  \textit{ in: Proceedings of Symposium in honor of Bruno Buchberger's 60th birthday "Logic, Mathematics and Computer Sciences: Intersections", Research Institute for Symbolic Computations, Linz, Austia, 2002, Report Series No. 02-60, 48-60 }

\item \textbf{Matthieu Calvez and Tetsuya Ito.} Garside-theoretic analysis of Burau representations. \textit{ arXiv:1401.2677v2 }

\item \textbf{D. D.Long and M. Paton.} The Burau representation is not faithful for 
$n \geq 6$. \textit{Topology}\ \textbf{32 }(1993), no. 2, 439---447\textbf{\textit{}}

\item \textbf{John Atwell  Moody.} The Burau representation of the braid group $B_n$ is unfaithful for large $ n$.  \textit{Bull. Amer. Math. Soc. (N.S.)} \textbf{25 }(1991) no. 2, 379--384\textbf{\textit{}}

\end{enumerate}
\noindent \textbf{\textit{}}
\noindent 

\noindent Authors' addresses:

\noindent Anzor Beridze,

\noindent Department of Mathematics,

\noindent Batumi ShotaRustaveli State University,

\noindent 35, Ninoshvili St., Batumi 6010,

\noindent Georgia

\noindent a.beridze@bsu.edu.ge;

\ \
\ \

\noindent Pawel Traczyk,

\noindent Institute of Mathematics,

\noindent University of Warsaw,

\noindent Banacha 2, 02-097 Warszawa,

\noindent Poland

\noindent traczyk@mimuw.edu.pl
\end{document}